\documentclass[11pt,a4paper]{article}
\usepackage{CJK}
\usepackage{indentfirst}
\usepackage[xdvi]{graphics}
\usepackage{latexsym,bm}
\usepackage{amssymb}
\usepackage{amsmath}
\usepackage{placeins}
\usepackage{amsthm}
\usepackage{color}

\CJKtilde
\newtheorem{thm}{Theorem}[section]

\newtheorem{lem}[thm]{Lemma}

\newtheorem{rem}[thm]{Remark}
\numberwithin{equation}{section}
\definecolor{red}{rgb}{1,0,0}

\begin{document}
\begin{CJK*}{GBK}{song}
\title{Lower bounds of eigenvalues of the biharmonic operators by the rectangular Morley element methods
\thanks{The first author was supported by the NSFC Project 11271035 and by the NSFC Key Project 11031006.}}
\author{\normalsize Jun Hu$^\dagger$,~~Xueqin Yang$^\dagger$ ~~ \\ \normalsize
$^\dagger$ LMAM and School of Mathematical Sciences,
Peking University, \\ \normalsize Beijing 100871, P.R.China
\\\vspace{3mm} \normalsize email: hujun@math.pku.edu.cn; ~~ yangxueqin1212@pku.edu.cn\\ \normalsize
}
\date{}
\maketitle
\CJKindent
\begin{abstract}
In this paper, we analyze the lower bound property of the discrete eigenvalues by the rectangular Morley elements of the biharmonic operators  in both  two and three dimensions. The analysis relies on an identity for the errors of eigenvalues. We  explore a refined property
 of the canonical interpolation operators and use it to analyze the key term in this identity. In particular, we show that such a term  is of higher order for two dimensions, and  is  negative and  of second order for three dimensions, which causes a main difficulty. To overcome it, we propose a novel decomposition of the first term in the aforementioned identity. Finally, we establish  a saturation condition  to  show that the discrete eigenvalues
  are smaller than the exact ones. We present some numerical results to demonstrate the theoretical results.
\end{abstract}

\section{Introduction}
We are interested in lower bounds of  the eigenvalue problem: Find $\lambda \in \mathbb{R}$ and $w \in V:=H_0^2(\Omega) $, such that
\begin{equation}
 a(w,v):=(\nabla^2 w,\nabla^2 v)_{L^2(\Omega)}=\lambda(w,v)_{L^2(\Omega)}\;\;\text{for  any}\;\; v \in V. \label{eq1}
\end{equation}
where $\nabla^2 w$ denotes the Hessian matrix of the function $w$. In 1979, Rannacher \cite{Rannacher} discovered
 by numerical results that both two dimensional Morley  and Adini elements eigenvalues could approximate the exact eigenvalues from below.  In 2000, Yang \cite{Yang} proved such a phenomenon  by showing that a consistent error defined there
 of the Adini element  is positive and of second order.  The analysis of  \cite{Yang} is based on a key fact that
  the order of convergence is  two for the Adini element eigenfunction in the energy norm.
  In 2012, Hu and Huang \cite{HuHuang} developed a correction operator for the canonical interpolation operators of  the Adini element  in both  two and three dimensions  and proved that the discrete eigenvalues are smaller than the exact ones, which employed an identity of the errors of eigenvalues due to \cite{ArmDuran} and \cite{ZhangYangChen}, see also \cite{HuHuangLin}.
  In particular, they showed that the last term in that identity is positive and of second order, which is based on the fact
  that the order of  convergence of the canonical interpolations of the Adini element is  two  in the energy norm.
    Besides, Hu, Huang and Lin \cite{HuHuangLin} proposed a new systematic method that can produce lower bounds for eigenvalues. In the same paper, they also showed that the Adini element satisfies the condition there and  consequently produces lower bounds of  eigenvalues.  We refer interested readers to \cite{HuangLiLin,LinLin,LinHuangLi} for the analysis, based on the expansion method,  of the lower bound property of discrete eigenvalues by  nonconforming rectangular elements of the Laplace operator in two dimensions. For other related works, we refer to \cite{YangLinBiLi} and \cite{YangZhangLin} and the references therein.

     The  purpose of the current paper  is to analyze the lower bound property of the discrete eigenvalues  obtained  by the rectangular Morley elements.  We shall follow \cite{HuHuang} and \cite{HuHuangLin} to use  the  identity  from \cite{ArmDuran,ZhangYangChen}. Note that the elements can not be analyzed by the theory of \cite{HuHuangLin}.  In addition, compared with  the Adini element analyzed in \cite{Yang} and \cite{HuHuang}, the
      main difficulties for the elements under consideration  are
      \begin{itemize}
   \item   the order of convergence is only one for both the discrete eigenfunctions and canonical interpolations  in the energy norm;
     \item  for the three dimensional element, the last term in the aforementioned identity is negative and of  second order.
      \end{itemize}
      To overcome these two difficulties, we use the expansion method proposed in \cite{HuShi2} to study a refined property
        of the canonical interpolation operators and propose a novel decomposition of the first term in the identity by using the canonical interpolation operators. Moreover, we prove a saturation condition and employ it to show that
        the discrete eigenvalues by the two and three dimensional rectangular Morley elements are smaller than the exact ones.

This paper is organized as follows. In the following section, we shall present the two-dimensional rectangular Morley element, and  show a refined property of the canonical interpolation operator and  use it  to prove that the discrete eigenvalues are smaller than the exact ones.  In section 3, we present the three-dimensional rectangular Morley element, show a refined property of the canonical interpolation operator,  and propose a novel decomposition of  the first
 term in the identity and   employ it, after establishing a saturation condition,  to prove the lower bound property of discrete eigenvalues. In section 4, we present some numerical results to demonstrate our theoretical results.

\section{Lower bounds of eigenvalues by the two-dimensional rectangular Morley element}

\subsection{The two-dimensional rectangular Morley element }

To consider the discretization of $(\ref{eq1})$ by the rectangular Morley element method, let $\mathcal{T}_h$ be a regular uniform rectangular triangulation of the domain $\Omega\subset \mathbb{R}^2$ in  two dimensions. Given $K\in \mathcal{T}_h$, let $(x_{1,c},x_{2,c})$ be the center of $K$, the meshsize $h$ and affine mapping:\\
\begin{equation} \label{eq18}  \xi_1=\frac{x_1-x_{1,c}}{h},\quad\xi_2=\frac{x_2-x_{2,c}}{h}\;\; \text{for any} \;(x_1,x_2)\in K.\end{equation}
On element $K$, the shape function space of the rectangular Morley element from \cite{ShiWang} reads\\
\begin{equation}
P_T(K):=P_2(K)+\text{span}\{x_1^3,x_2^3\},
\end{equation}
here and throughout this paper, $P_l(K)$ denotes the space of polynomials of degree $\leq l$ over $K$. The nodal parameters are: for any $v\in C^1(K)$,
\begin{equation}
D_T(v)=\bigg(v(a_i),\;\;\frac{1}{|F_j|}\int_{F_j} \frac{\partial v}{\partial \nu_{F_j}}\;\mathrm{d}s\bigg),\;\;i,j=1,2,3,4, \label{eq32}
\end{equation}
where $a_i$ are vertices of $K$ and $F_j$ are edges of $K$. $|F_j|$ denote measure of edges $F_j$, see Figure 1.
\begin{figure}
%\label{fig1}
\begin{center}
\setlength{\unitlength}{2cm}
\begin{picture}(2,2)
\put(0,0.5){\line(1,0){2}} \put(0,0.5){\line(0,1){1}}
\put(2,0.5){\line(0,1){1}} \put(0,1.5){\line(1,0){2}}
\put(0,0.5){\circle*{0.1}}
\put(0,1.5){\circle*{0.1}}
\put(2,0.5){\circle*{0.1}}
\put(2,1.5){\circle*{0.1}}

\put(1,0.5){\vector(0,-1){0.3}}
\put(1,0.1){$\int$}\put(1,1.8){$\int$} \put(-0.4,1){$\int$} \put(2.3,1){$\int$}
\put(1,1.5){\vector(0,1){0.3}}
\put(0,1){\vector(-1,0){0.3}}
\put(2,1){\vector(1,0){0.3}}

\linethickness{0.6mm}
\end{picture}
\end{center}
\caption{degrees of freedom }
\end{figure}
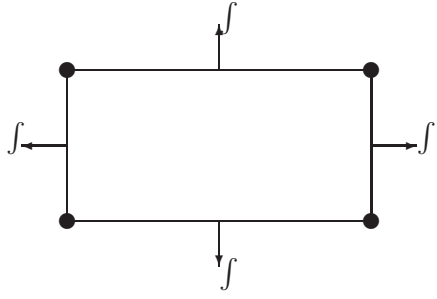
Corresponding to the nodal parameters, the basis functions are the same as those which can be found in \cite{ShiWang}. The $P_T$-unisolvence of $D_T$ can be proved similarly as \cite{ShiWang}.
The nonconforming rectangular Morley element space is then defined by
\begin{eqnarray*}
V_h:&=&\{v \in L^2(\Omega):\;v|_K \in P_T(K),\; \forall K \in \mathcal{T}_h, \;v\;is\;continuous\;at\;all\;internal\;vertices\;and\\&&\;vanishes\;at\;all\;boundary\;vertices,\;and\int_F \frac{\partial v}{\partial \nu_F}\,\mathrm{d}s\;is\;continuous\;at\;internal\;edges\;\\&& F\;and\;vanishes\;at\;boundary\;edges\; of \;\mathcal{T}_h\}.
\end{eqnarray*}
The finite element approximation of Problem (\ref{eq1}) reads: Find $\lambda_h \in \mathbb{R}$ and $w_h \in V_h$, such that
\begin{equation}
a_h(w_h,v_h):=(\nabla_h^2 w_h,\nabla_h^2 v_h)_{L^2(\Omega)}=\lambda_h(w_h,v_h)_{L^2(\Omega)}\;\; \text{for any}\; v_h \in V_h. \label{eq45}
\end{equation}
where the operator $\nabla_h^2$ is the discrete counterpart of $\nabla^2$, which is defined element by element since the discrete space $V_h$ is nonconforming.
\subsection{Interpolation operators}
Given $K\in \mathcal{T}_h$, define the interpolation operator $\Pi_K : H^3(K)\rightarrow P_T(K)$ by, for any $v \in H^3(K) $,
\begin{equation} \label{eq4}  (\Pi_K v)(P)=v(P)\; and\;\int_F \frac{\partial \Pi_K v}{\partial \nu_F}\;\mathrm{d}s=\int_F \frac{\partial v}{\partial \nu_F}\;\mathrm{d}s, \end{equation}\\
for any vertex $P$ of $K$ and any edge $F$ of $K$ in the two-dimensional case or any face $F$ of $K$ in the three-dimensional case. The interpolation operator $\Pi_K$ has the following error estimates:
\begin{equation}\label{eq5}  |v-\Pi_K v|_{H^l(K)}\leq Ch^{3-l}|v|_{H^3(K)},\;l=0,1,2,3, \end{equation}\\
provided that $v \in H^3(K)$. Herein and throughout this paper, $C$ denotes a generic positive constant which is independent of the meshsize and may be different at different places. Then the global version $\Pi_h$ of the interpolation operator $\Pi_K$ is defined as
\begin{equation}\label{eq6}  \Pi_h|_K=\Pi_K\;\; \text{for any}\; K \in \mathcal{T}_h. \end{equation}\\
In the following, let $\nabla^l v$ denote the $l$-$th$ order tensor of all the $l$-$th$ order derivatives of $v$ and $\nabla_h^l$ denote the piecewise counterpart of $\nabla^l$ defined element by element.

Given any element $K$, we follow \cite{HuShi1} to define $P_K v \in P_4(K)$ by
\begin{equation}
 \int_K\nabla^l P_K v\,\mathrm{d}x_1\mathrm{d}x_2=\int_K\nabla^l v\,\mathrm{d}x_1\mathrm{d}x_2,\;l=0,1,2,3,4, \label{eq56}
\end{equation}
for two dimensions, and \begin{equation}\label{eq7} \int_K\nabla^l P_K v\,\mathrm{d}x_1\mathrm{d}x_2\mathrm{d}x_3=\int_K\nabla^l v\,\mathrm{d}x_1\mathrm{d}x_2\mathrm{d}x_3,\;l=0,1,2,3,4,
\end{equation}
for three dimensions, for any $v \in H^4(K).$ Note that the operator $P_K$ is well-defined. The interpolation operator $P_K$ has the following error estimates:
\begin{eqnarray}
 |v-P_K v|_{H^j(K)}&\leq& Ch^{4-j}|v|_{H^4(K)},\;j=0,1,2,3,4,\nonumber\\
 |v-P_K v|_{H^j(K)}&\leq& Ch|v|_{H^{j+1}(K)},\;j=0,1,2,3,\label{eq34}
 \end{eqnarray}\\
provided that $v \in H^4(K)$. It follows from the definition of $P_K$ in (\ref{eq7}) that
\begin{equation}
\nabla^4 P_K v=\Pi_{0,K}\nabla^4 v, \label{eq8}
\end{equation}
with $\Pi_{0,K}$ the $L^2$ constant projection operator over $K$. The global version $\Pi_0$ of the interpolation operator $\Pi_{0,K}$ is defined as
\begin{equation}\label{eq9} \Pi_0|_K=\Pi_{0,K}\;\text{for any}\;K \in \mathcal{T}_h. \end{equation}
\subsection{Lower bounds of eigenvalues}
In this section, we are in the position to show that the approximate eigenvalues are smaller than the exact ones.
Define a semi-norm over $V_h$ by
\begin{equation*}|u_h|_h^2:=a_h(u_h,u_h)\; \;\;\text{for any}\; u_h\in V_h. \end{equation*}
By the error estimates of the interpolation $\Pi_h u$ and the finite element solution $u_h$ of the eigenfunction $u$, it follows from \cite{ShiWang} and \cite{HuHuangLin,HuShi1} that
\begin{eqnarray}
||u-u_h||_{L^2(\Omega)}\leq C h^2,\quad |u-u_h|_h^2\leq C h^2. \label{eq47}
\end{eqnarray}
By the triangle inequality and (\ref{eq5}), (\ref{eq47}), we can get that
\begin{equation}
||\Pi_h u-u_h||_{L^2(\Omega)}\leq||u-u_h||_{L^2(\Omega)}+||u-\Pi_h u||_{L^2(\Omega)}\leq C h^2.\label{eq51}
\end{equation}
\begin{thm}
Let $(\lambda,u)$ and $(\lambda_h,u_h)$ be the solutions to (\ref{eq1}) and (\ref{eq45}), respectively, and assume that $u \in H_0^2(\Omega)\bigcap H^4(\Omega)$,\;then,
\[\lambda_h \leq \lambda,\]
provided that $h$ is small enough.
\end{thm}
\begin{proof}
We need to use an identity for the errors of the eigenvalues from \cite{ArmDuran,ZhangYangChen}, see also \cite{HuHuangLin}.
\begin{eqnarray}
\lambda-\lambda_h &=&|u-u_h|_h^2-\lambda_h(\Pi_h u-u_h,\Pi_h u-u_h)_{L^2(\Omega)}\nonumber\\&&
                    +\lambda_h\big(||\Pi_h u||_{L^2(\Omega)}^2-||u||_{L^2(\Omega)}^2\big)+2a_h(u-\Pi_h u,u_h). \label{eq37}
\end{eqnarray}
We can bound the second term by (\ref{eq51})
\begin{equation}
(\Pi_h u-u_h,\Pi_h u-u_h)_{L^2(\Omega)} \leq C h^4
\end{equation}
and the third term as
\begin{eqnarray}
|||\Pi_h u||_{L^2(\Omega)}^2-||u||_{L^2(\Omega)}^2|&=&|(\Pi_h u,\Pi_h u)_{L^2(\Omega)}-(u,u)_{L^2(\Omega)}|\nonumber\\
                                               &=&|(\Pi_h u-u,\Pi_h u)_{L^2(\Omega)}+(u,\Pi_h u-u)_{L^2(\Omega)}|\nonumber\\
                                               &\leq&Ch^3|u|_{H^3(\Omega)}\big(||\Pi_h u||_{L^2(\Omega)}+||u||_{L^2(\Omega)}\big)\nonumber\\
                                               &\leq&Ch^3|u|_{H^3(\Omega)}\big(||\Pi_h u-u||_{L^2(\Omega)}+2||u||_{L^2(\Omega)}\big)\nonumber\\
                                               &\leq&Ch^3|u|_{H^3(\Omega)}^2.
\end{eqnarray}
Herein, we provide a new method to deal with the fourth term. A combination of the first term and the fourth term of (\ref{eq37}) has the following decomposition
\begin{eqnarray}
|u-u_h|^2_h+2a_h(u-\Pi_h u,u_h)&=&
(\nabla_h^2(u-\Pi_h u),\nabla_h^2(u-\Pi_h u))_{L^2(\Omega)}\nonumber\\&&+(\nabla_h^2(\Pi_h u-u_h),\nabla_h^2(\Pi_h u-u_h))_{L^2(\Omega)}\nonumber\\&&+2(\nabla_h^2(u-\Pi_h u),\nabla_h^2(\Pi_h u-u_h)_{L^2(\Omega)}\nonumber\\&&+2(\nabla_h^2(u-\Pi_h u),\nabla_h^2u_h)_{L^2(\Omega)}\nonumber\\&=&||\nabla_h^2(u-\Pi_h u)||^2_{L^2(\Omega)}+||\nabla_h^2(\Pi_h u-u_h)||^2_{L^2(\Omega)}\nonumber\\&&+2(\nabla_h^2(u-\Pi_h u),\nabla_h^2\Pi_h u)_{L^2(\Omega)}. \label{eq80}
\end{eqnarray}
The third term of (\ref{eq80}) will be analyzed in Lemma \ref{lem2} below, which reads
\begin{equation}
(\nabla_h^2(u-\Pi_h u),\nabla_h^2\Pi_h u)_{L^2(\Omega)}\leq \alpha_h h^2, \label{eq48}
\end{equation}
where $\lim\limits_{h\rightarrow 0}\alpha_h=0.$
It follows from \cite{HuHuangLin} that there holds the saturation condition $h^2\leq C |u-u_h|_h^2$.
Hence, by (\ref{eq47})$-$(\ref{eq48}), the sign of $\lambda-\lambda_h$ is dominated by the first term and the second term of (\ref{eq80}).
\end{proof}
\subsection{A refined property of the interpolation operator }
Given $K\in \mathcal{T}_h$, for ease of presentation, we define seven bubble functions with respect to the degrees of freedom defined as in (\ref{eq32}) as follows
 \begin{eqnarray}
 \begin{array}{lll}
 \varphi_{i,j}=\xi_i^2\xi_j-\frac{4}{3}\xi_j+\frac{\xi_j^3}{3},\quad i,j=1,2,\;i\neq j,\\
 \psi_i=(\xi_i^2-1)^2,\quad i=1,2,\\
 p=\xi_1^2+\xi_2^2-\frac{\xi_1^3}{3}-\frac{\xi_2^3}{3}-\frac{1}{3},\\
 q_{i,j}=\xi_i^3\xi_j-\xi_i\xi_j,\quad i,j=1,2,\,i\neq j,\label{eq40}
 \end{array}
 \end{eqnarray}
In fact, it can be checked that
 \begin{eqnarray*}
 &&\left\{
 \begin{array}{lll}
 \varphi_{i,j}(a_k)=0,\quad i,j=1,2,\,i\neq j,\,k=1,2,3,4,\\
 \int_{F_k} \frac{\partial\varphi_{i,j}}{\partial\nu_{F_k}}\;\mathrm{d}s=0,\quad i,j=1,2,\,i\neq j,\,k=1,2,3,4,\\
 \end{array}
 \right.\\
 &&\left\{
 \begin{array}{ccc}
 \psi_i(a_k)=0,\quad i=1,2,\,k=1,2,3,4,\\
 \int_{F_k} \frac{\partial\psi_i}{\partial\nu_{F_k}}\;\mathrm{d}s=0,\quad i=1,2,\,k=1,2,3,4,\\
 \end{array}
 \right.\\
 &&\left\{
 \begin{array}{lll}
 p(a_k)=0,\quad k=1,2,3,4,\\
 \int_{F_k} \frac{\partial p}{\partial\nu_{F_k}}\;\mathrm{d}s=0,\quad k=1,2,3,4,\\
 \end{array}
 \right.\\
 &&\left\{
 \begin{array}{lll}
 q_{i,j}(a_k)=0,\quad i,j=1,2,\,i\neq j,\,k=1,2,3,4,\\
 \int_{F_k} \frac{\partial q_{i,j}}{\partial\nu_{F_k}}\;\mathrm{d}s=0,\quad i,j=1,2,\,i\neq j,\,k=1,2,3,4,
 \end{array}
 \right.
 \end{eqnarray*}
 where $a_k$ are vertices of $K$, and $F_k$ are edges of $K$ .\\

In the next lemma, we follow the idea of \cite{HuShi1} to analyze a new refined property for the interpolation operator, which is a basis for the analysis of the term $a_h(u-\Pi_h u,u_h)$.

\begin{lem} \label{lem1}
Given $K\in \mathcal{T}_h$, for any $u\in P_4(K)$ and $ v\in P_T(K)$, there holds that
 \begin{eqnarray*}
(\nabla^2(u-\Pi_Ku),\nabla^2v)_{L^2(K)}&=&\frac{h^2}{3}\int_K\frac{\partial^3u}{\partial x_1\partial x_2^2}\frac{\partial^3v}{\partial x_1^3}\,\mathrm{d}x_1\mathrm{d}x_2\\&&+\frac{h^2}{3}\int_K\frac{\partial^3u}{\partial x_1^2\partial x_2}\frac{\partial^3v}{\partial x_2^3}\,\mathrm{d}x_1\mathrm{d}x_2.
\end{eqnarray*}.
\end{lem}
\begin{proof}
Let $\xi_1$ and $\xi_2 $  be defined as in $(\ref{eq18})$. It follows from the definition of $P_T(K)$ that\\
\begin{eqnarray}
\frac{\partial^2v}{\partial x_i^2}&=&\overline{\frac{\partial^2v}{\partial x_i^2}}+h\frac{\partial^3v}{\partial x_i^3}\xi_i,\quad i=1,2,\\
\frac{\partial^2v}{\partial x_1\partial x_2}&=&\overline{\frac{\partial^2v}{\partial x_1\partial x_2}}, \label{eq21}
\end{eqnarray}
where $\overline{f}$ denotes the integral average of $f$ over $K$, namely, \begin{equation*}\overline{f}=\frac{1}{|K|}\int_K f\mathrm{d}x_1\mathrm{d}x_2.\end{equation*} Since $u\in P_4(K)$, the Taylor expansion and the definition of the operator $\Pi_K$ yield\\
\begin{eqnarray}
u-\Pi_K u&=&\frac{h^3}{2!}\sum_{i\neq j=1}^2\overline{\frac{\partial^3u}{\partial x_i^2\partial x_j}}\varphi_{i,j}+
          \frac{h^4}{4!}\sum_{i=1}^2\frac{\partial^4u}{\partial x_i^4}\psi_i\nonumber\\&&+
          \frac{h^4}{2!2!}\frac{\partial^4u}{\partial x_1^2\partial x_2^2}p+
          \frac{h^4}{3!}\sum_{i\neq j=1}^2\frac{\partial^4u}{\partial x_i^3\partial x_j}q_{i,j},\label{eq57}
\end{eqnarray}
where $\varphi_{i,j}$, $\psi_{i,j}$, $p$, and $q_{i,j}$ are defined as in (\ref{eq40}).
Hence, the second order partial derivative of $u-\Pi_K u$ with respect to the variable $x_1$ reads
\begin{eqnarray}
\frac{\partial^2(u-\Pi_Ku)}{\partial x_1^2}&=&\frac{h}{2!}\overline{\frac{\partial^3u}{\partial x_1^2\partial x_2}}2\xi_2+
                                          \frac{h}{2!}\overline{\frac{\partial^3u}{\partial x_1\partial x_2^2}}2\xi_1\nonumber\\&&+
                                          \frac{h^2}{4!}\frac{\partial^4u}{\partial x_1^4}(12\xi_1^2-4)+
                                          \frac{h^2}{2!2!}\frac{\partial^4u}{\partial x_1^2\partial x_2^2}(2\xi_2^2-\frac{2}{3})\nonumber\\&&+
                                          \frac{h^2}{3!}\frac{\partial^4u}{\partial x_1^3\partial x_2}6\xi_1\xi_2.\label{eq24}
\end{eqnarray}
A combination of $(\ref{eq21})$ and $(\ref{eq24})$ plus some elementary calculation gives
\begin{eqnarray*}
\int_K\frac{\partial^2(u-\Pi_Ku)}{\partial x_1^2}\frac{\partial^2v}{\partial x_1^2}\,\mathrm{d}x_1\mathrm{d}x_2&=&
                \int_K \bigg(\frac{h}{2!}\overline{\frac{\partial^3u}{\partial x_1^2\partial x_2}}2\xi_2+
                                          \frac{h}{2!}\overline{\frac{\partial^3u}{\partial x_1\partial x_2^2}}2\xi_1+
                                          \frac{h^2}{4!}\frac{\partial^4u}{\partial x_1^4}(12\xi_1^2-4)
                                          +\frac{h^2}{2!2!}\\&&\frac{\partial^4u}{\partial x_1^2\partial x_2^2}(2\xi_2^2-\frac{2}{3})+
                                          \frac{h^2}{3!}\frac{\partial^4u}{\partial x_1^3\partial x_2}6\xi_1\xi_2\bigg)\Big(\overline{\frac{\partial^2v}{\partial x_1^2}}+h\frac{\partial^3v}{\partial x_1^3}\xi_1\Big)\;\mathrm{d}x_1\mathrm{d}x_2.
\end{eqnarray*}
Since all coefficients like $\overline{\frac{\partial^3u}{\partial x_1^2\partial x_2}}$ and $\frac{\partial^4 u}{\partial x_1^4}$ are constants,  we can get that by  parity of functions and symmetry of domains:
\begin{eqnarray*}
\int_K 2h\xi_2\;\mathrm{d}x_1\mathrm{d}x_2=0,\;\int_K 2h\xi_1\;\mathrm{d}x_1\mathrm{d}x_2=0,\;\int_K h^2(12\xi_1^2-4)\;\mathrm{d}x_1\mathrm{d}x_2=0,\\
\int_K h^2(2\xi_2^2-\frac{2}{3})\;\mathrm{d}x_1\mathrm{d}x_2=0,\;\int_K h^26\xi_1\xi_2\;\mathrm{d}x_1\mathrm{d}x_2=0.
\end{eqnarray*}
Hence, only one nonzero term is left, which reads
\begin{eqnarray*}
\int_K \frac{h}{2!}\overline{\frac{\partial^3u}{\partial x_1\partial x_2^2}}2\xi_1h\frac{\partial^3v}{\partial x_1^3}\xi_1\;\mathrm{d}x_1\mathrm{d}x_2&=&\frac{h^2}{3}\int_K\frac{\partial^3u}{\partial x_1\partial x_2^2}\frac{\partial^3v}{\partial x_1^3}\,\mathrm{d}x_1\mathrm{d}x_2.\\
\end{eqnarray*}
This yields
\begin{eqnarray*}
\int_K\frac{\partial^2(u-\Pi_Ku)}{\partial x_1^2}\frac{\partial^2v}{\partial x_1^2}\,\mathrm{d}x_1\mathrm{d}x_2&=&
                 \frac{h^2}{3}\int_K\frac{\partial^3u}{\partial x_1\partial x_2^2}\frac{\partial^3v}{\partial x_1^3}\,\mathrm{d}x_1\mathrm{d}x_2.
\end{eqnarray*}
A similar argument proves
\begin{eqnarray*}
\int_K\frac{\partial^2(u-\Pi_Ku)}{\partial x_2^2}\frac{\partial^2v}{\partial x_2^2}\,\mathrm{d}x_1\mathrm{d}x_2&=&
                 \frac{h^2}{3}\int_K\frac{\partial^3u}{\partial x_1^2\partial x_2}\frac{\partial^3v}{\partial x_2^3}\,\mathrm{d}x_1\mathrm{d}x_2.
\end{eqnarray*}
Finally, the second order mixed partial derivative of $u-\Pi_K u$ is
\begin{eqnarray}
\frac{\partial^2(u-\Pi_Ku)}{\partial x_1\partial x_2}&=&\frac{h}{2!}\overline{\frac{\partial^3u}{\partial x_1^2\partial x_2}}2\xi_1\nonumber+
                                                   \frac{h}{2!}\overline{\frac{\partial^3u}{\partial x_1\partial x_2^2}}2\xi_2\nonumber\\&&+
                                                   \frac{h^2}{2!2!}\frac{\partial^4u}{\partial x_1^2\partial x_2^2}4\xi_1\xi_2\nonumber+
                                                   \frac{h^2}{3!}\frac{\partial^4u}{\partial x_1^3\partial x_2}(3\xi_1^2-1)\nonumber\\&&+
                                                   \frac{h^2}{3!}\frac{\partial^4u}{\partial x_1\partial x_2^3}(3\xi_2^2-1). \label{eq22}
\end{eqnarray}
 A similar procedure of the first part of the proof, this and (\ref{eq21}) lead to
\begin{eqnarray}
\int_K\frac{\partial^2(u-\Pi_Ku)}{\partial x_1\partial x_2}\frac{\partial^2v}{\partial x_1\partial x_2}\,\mathrm{d}x_1\mathrm{d}x_2&=&0,\label{eq12}
\end{eqnarray}
which completes the proof.
\end{proof}
Next, we use Lemma \ref{lem1}  to analyze the key term in the proof of  Theorem 2.1.
\begin{lem} \label{lem2}
Suppose that $w \in H_0^2(\Omega) \bigcap H^4(\Omega)$ with $\Omega\subset \mathbb{R}^2$. Then,
\begin{equation}
 (\nabla_h^2(w-\Pi_h w),\nabla_h^2 \Pi_h w)_{L^2(\Omega)}\leq\alpha_h h^2,\end{equation}
 where $\lim\limits_{h\rightarrow 0}\alpha_h=0.$

\end{lem}
\begin{rem}
For the Adini element, this term is positive and of order $O(h^2)$. However, for the rectangular Morley element, we cannot get a similar result as the Adini element, which indicates a  difficulty for the analysis herein.
\end{rem}
\begin{proof}
Given $K \in \mathcal{T}_h$, let the interpolation operator $P_K$ be defined as in (\ref{eq56}), which leads to the following decomposition
\begin{eqnarray}
(\nabla_h^2(w-\Pi_h w),\nabla_h^2\Pi_h w)_{L^2(\Omega)}&=&\sum_{K\in\mathcal{T}_h} (\nabla_h^2(P_Kw-\Pi_K P_K w),\nabla_h^2\Pi_K w)_{L^2(K)}\nonumber\\
                                                         &&+\sum_{K\in\mathcal{T}_h} (\nabla_h^2(I-\Pi_K)(I-P_K)w,\nabla_h^2\Pi_K w)_{L^2(K)}\nonumber\\
                                                       &=&I_1+I_2.  \label{eq13}
\end{eqnarray}

We first analyze the first term $I_1$ on the right-hand side of (\ref{eq13}).
Let $u=P_Kw$ and $v=\Pi_Kw$ in Lemma \ref{lem1}. The first term $I_1$ on the right-hand side of  (\ref{eq13}) can be rewritten as
\begin{eqnarray*}
I_1&=&\sum_{i\neq j=1}^2\sum_{K\in\mathcal{T}_h} \frac{h^2}{3}\int_K\frac{\partial^3P_Kw}{\partial x_i\partial x_j^2}\frac{\partial^3\Pi_Kw}{\partial x_i^3}\,\mathrm{d}x_1\mathrm{d}x_2=I_{1,1}+I_{1,2}.
\end{eqnarray*}
It is straightforward to show that the first term of $I_1$ can be expressed as
\begin{eqnarray}
I_{1,1}&=&\sum_{K\in\mathcal{T}_h} \frac{h^2}{3}\int_K\frac{\partial^3P_Kw}{\partial x_1\partial x_2^2}\frac{\partial^3\Pi_Kw}{\partial x_1^3}\,\mathrm{d}x_1\mathrm{d}x_2\nonumber\\&=&
\sum_{K\in\mathcal{T}_h} \frac{h^2}{3}\int_K\frac{\partial^3w}{\partial x_1\partial x_2^2}\frac{\partial^3w}{\partial x_1^3}\,\mathrm{d}x_1\mathrm{d}x_2+
\sum_{K\in\mathcal{T}_h} \frac{h^2}{3}\int_K\frac{\partial^3(P_K-I)w}{\partial x_1\partial x_2^2}\frac{\partial^3\Pi_Kw}{\partial x_1^3}\,\mathrm{d}x_1\mathrm{d}x_2\nonumber\\&&+
\sum_{K\in\mathcal{T}_h} \frac{h^2}{3}\int_K\frac{\partial^3w}{\partial x_1\partial x_2^2}\frac{\partial^3(\Pi_Kw-w)}{\partial x_1^3}\,\mathrm{d}x_1\mathrm{d}x_2\nonumber\\&=&I_{1,1}^{(1)}+I_{1,1}^{(2)}+I_{1,1}^{(3)}.\label{eq33}
\end{eqnarray}

We are in the position to estimate three terms on the right-hand side of (\ref{eq33}).
Integrating by parts twice and using the fact that
\begin{equation*}
\frac{\partial^2 w}{\partial x_1\partial x_2}\frac{\partial^3 w}{\partial x_1^3}\bigg|_{\partial \Omega}=\frac{\partial^2 w}{\partial x_1\partial x_2}\frac{\partial^3 w}{\partial x_1^2\partial x_2}\bigg|_{\partial \Omega}=0
\end{equation*}
show that the first term $I_{1,1}^{(1)}$ on the right-hand side of (\ref{eq33}) is
\begin{eqnarray*}
I_{1,1}^{(1)}=\sum_{K\in\mathcal{T}_h} \frac{h^2}{3}\int_K\frac{\partial^3w}{\partial x_1\partial x_2^2}\frac{\partial^3w}{\partial x_1^3}\,\mathrm{d}x_1\mathrm{d}x_2=\frac{h^2}{3}||\frac{\partial^3w}{\partial x_1^2\partial x_2}||^2_{L^2(\Omega)}.
\end{eqnarray*}
Since $||\frac{\partial^3\Pi_K w}{\partial x_1^3}||_{L^2(K)}\leq C|w|_{H^3(K)}$ is bounded, it follows from (\ref{eq34}) that
\begin{eqnarray*}
I_{1,1}^{(2)}=\sum_{K\in\mathcal{T}_h} \frac{h^2}{3}\int_K\frac{\partial^3(P_K-I)w}{\partial x_1\partial x_2^2}\frac{\partial^3\Pi_Kw}{\partial x_1^3}\,\mathrm{d}x_1\mathrm{d}x_2\leq Ch^3|w|_{H^4(\Omega)}^2.
\end{eqnarray*}
For the Adini element, the third term $I_{1,1}^{(3)}$ is a higher order term, since its shape function space contains $P_3(K)$. Herein, to analyze the third term $I_{1,1}^{(3)}$, we need the following expansion, see (\ref{eq57}), up to a higher order term,
\begin{equation*}
(w-\Pi_Kw)\big|_K=\frac{h^3}{2!}\sum_{i\neq j=1}^2\overline{\frac{\partial^3w}{\partial x_i^2\partial x_j}}\varphi_{i,j}+Ch^4.
%w-\Pi_Kw=\frac{h^3}{2!}\frac{\partial^3w}{\partial x_1^2\partial x_2}(\xi_1^2\xi_2-\frac{4}{3}\xi_2+\frac{\xi_2^3}{3})+\frac{h^3}{2!}\frac{\partial^3w}{\partial x_1\partial x_2^2}(\xi_1 \xi_2^2-\frac{4}{3}\xi_1+\frac{\xi_1^3}{3})+Ch^4.
\end{equation*}
Since $\frac{\partial^3 \varphi_{1,2}}{\partial x_1^3}\big|_K=0$ and $\frac{\partial^3 \varphi_{2,1}}{\partial x_1^3}\big|_K=2h^{-3}$, it follows that
\begin{eqnarray*}
I_{1,1}^{(3)}&=&\sum_{K\in\mathcal{T}_h} \frac{h^2}{3}\int_K\frac{\partial^3w}{\partial x_1\partial x_2^2}\frac{\partial^3(\Pi_Kw-w)}{\partial x_1^3}\,\mathrm{d}x_1\mathrm{d}x_2\\&=&-\frac{h^2}{3} \sum_{K\in\mathcal{T}_h}\int_K\bigg(\frac{\partial^3w}{\partial x_1\partial x_2^2}\bigg)^2\mathrm{d}x_1\mathrm{d}x_2+Ch^3.
\end{eqnarray*}
A summary of these previous three equations yields
\begin{eqnarray*}
I_{1,1}&=&I_{1,1}^{(1)}+I_{1,1}^{(2)}+I_{1,1}^{(3)}\\&=&\frac{h^2}{3}||\frac{\partial^3w}{\partial x_1^2\partial x_2}||^2_{L^2(\Omega)}-\frac{h^2}{3} \sum_{K\in\mathcal{T}_h}\int_K\bigg(\frac{\partial^3w}{\partial x_1\partial x_2^2}\bigg)^2\mathrm{d}x_1\mathrm{d}x_2+Ch^3.
\end{eqnarray*}
A similar analysis applies to the second term of $I_1$, which implies
\begin{eqnarray*}
I_{1,2}&=&\frac{h^2}{3}||\frac{\partial^3w}{\partial x_1\partial x_2^2}||^2_{L^2(\Omega)}-\frac{h^2}{3} \sum_{K\in\mathcal{T}_h}\int_K\bigg(\frac{\partial^3w}{\partial x_1^2\partial x_2}\bigg)^2\mathrm{d}x_1\mathrm{d}x_2+Ch^3.
\end{eqnarray*}
This leads to
\begin{equation}
 I_1=I_{1,1}+I_{1,2}=0+Ch^3.  \label{eq35}
\end{equation}
We turn to the second term $I_2$ on the right-hand side of (\ref{eq13}), which  can be estimated by the error estimates of (\ref{eq34}) as
\begin{equation*}
|I_2|\leq Ch \sum_{K\in\mathcal{T}_h}||\nabla_h^3(I-P_K)w||_{L^2(K)}|w|_{H^3(\Omega)}.
\end{equation*}
The definition of the projection operator $P_K$ gives
\begin{equation*}
\int_K \nabla^3(I-P_K)w\;\mathrm{d}x_1\mathrm{d}x_2=0.
\end{equation*}
By the  Poincare inequality, and the commuting property of (\ref{eq8}),
\begin{eqnarray}
|I_2|&\leq & Ch^2 \sum_{K\in\mathcal{T}_h}||\nabla_h^4(I-P_K)w||_{L^2(K)}|w|_{H^3(\Omega)}\nonumber\\&\leq&Ch^2||(I-\Pi_0)\nabla_h^4 w||_{L^2(\Omega)}|w|_{H^3(\Omega)}.\label{eq25}
\end{eqnarray}
Since the piecewise constant functions are dense in the space $L^2(\Omega)$,
\begin{equation}
||(I-\Pi_0)\nabla_h^4 w||_{L^2(\Omega)}\rightarrow\; 0,\;\;when\;\;h\rightarrow \;0. \label{eq31}
\end{equation}
A summary of (\ref{eq35}), (\ref{eq25}) and (\ref{eq31}) completes the proof.
\end{proof}
\begin{rem}
Comparing with the Adini element for the fourth order eigenvalue problem \cite{HuHuang}, the proof herein weakens the regularity condition from $H^{4+s}$ with $(0<s\leq1)$ to $H^4$.
\end{rem}

\section{Lower bounds of eigenvalues by the three-dimensional rectangular Morley element }
The section also  uses the identity for the errors of eigenvalue from \cite{ArmDuran,ZhangYangChen}, see also \cite{HuHuangLin}. However, for the three dimensional case, the last term on the right hand side of (\ref{eq15}) is negative and of order $O(h^2)$. This causes a main difficulty. To overcome this difficulty, we propose a new decomposition of the first term by using the canonical interpolation operator defined as in (\ref{eq4}), see more details in (\ref{eq19}) below.
\subsection{The three-dimensional rectangle Morley element }
Let $\mathcal{T}_h$ be a regular uniform rectangular triangulation of the domain $\Omega\subset \mathbb{R}^3$ in three dimensions. Given $K\in \mathcal{T}_h$, let $(x_{1,c},x_{2,c},x_{3,c})$ be the center of $K$, the meshsize $h$ and affine mapping:\\
\begin{equation} \label{eq2}  \xi_1=\frac{x_1-x_{1,c}}{h},\quad\xi_2=\frac{x_2-x_{2,c}}{h},\quad\xi_3=\frac{x_3-x_{3,c}}{h}\;\; \text{for any}\; (x_1,x_2,x_3)\in K.\end{equation}
On element $K$, the shape function space of the rectangular Morley element reads\\
\begin{equation}   P_T(K):=P_2(K)+\text{span}\{x_1^3,x_2^3,x_3^3,x_1x_2x_3\}.\end{equation}
The nodal parameters are: for any $v\in C^1(K)$,
\begin{equation}
D_T(v)=\bigg(v(a_i),\;\;\frac{1}{|F_j|}\int_{F_j} \frac{\partial v}{\partial \nu_{F_j}}\;\mathrm{d}s\bigg),\;\;i=1,\ldots,8,\;j=1,\ldots,6, \label{eq55}
\end{equation}
\begin{figure}
%\label{fig2}
\begin{center}
\setlength{\unitlength}{2.5cm}
\begin{picture}(4,2)

%% 3D
\put(1,0.5){\line(1,0){2}} \put(1,0.5){\line(0,1){1}}
\put(3,0.5){\line(0,1){1}} \put(1,1.5){\line(1,0){2}}
\put(1,0.5){\circle*{0.1}}
\put(1,1.5){\circle*{0.1}}
\put(3,0.5){\circle*{0.1}}
\put(3,1.5){\circle*{0.1}}
\put(1,1.5){\line(1,1){0.3}}
\put(3,1.5){\line(1,1){0.3}}
\put(1.3,1.8){\line(1,0){2}}
\put(1.3,1.8){\circle*{0.1}}
\put(3.3,1.8){\circle*{0.1}}
\put(3,0.5){\line(1,1){0.3}}
\put(3.3,0.8){\circle*{0.1}}
\put(3.3,0.8){\line(0,1){1}}
\multiput(1,0.5)(0.18,0.18){2}{\line(1,1){0.15}}
\multiput(1.3,0.8)(0.14,0){15}{\line(1,0){0.1}}
\multiput(1.3,0.8)(0,0.14){7}{\line(0,1){0.1}}
\put(1.3,0.8){\circle*{0.1}}
\put(2.2,0.65){\vector(0,-1){0.3}}
\put(2.2,1.65){\vector(0,1){0.3}}
\put(1.15,1.15){\vector(-1,0){0.3}}
\put(3.15,1.15){\vector(1,0){0.3}}
\put(2,1){\vector(-1,-1){0.25}}
\put(2.3,1.3){\vector(1,1){0.25}}
\put(2.2,0.25){$\iint$}
\put(2.2,1.95){$\iint$}
\put(0.65,1.15){$\iint$}
\put(3.45,1.15){$\iint$}
\put(1.5,0.7){$\iint$}
\put(2.6,1.6){$\iint$}
\linethickness{0.6mm}
\end{picture}
\end{center}
\caption{degrees of freedom }
\end{figure}
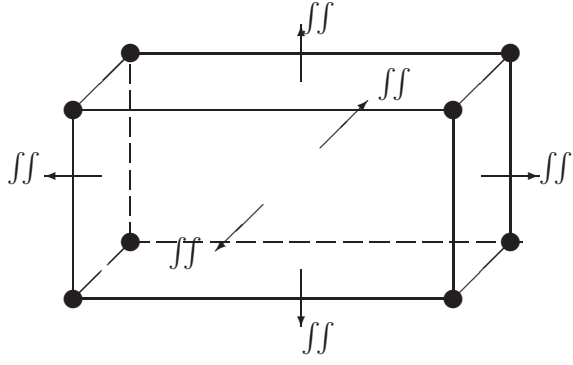
where $a_i$ are vertices of $K$ and $F_j$ are faces of $K$, see Figure 2. The $P_T$-unisolvence of  $D_T$ can be found in \cite{ShiWang}. The nonconforming rectangular Morley element space is then defined by
\begin{eqnarray*}
V_h:&=&\{v \in L^2(\Omega):\;v|_K \in P_T(K), \forall K \in \mathcal{T}_h, v\;is\;continuous\;at\;all\;internal\;vertices\; and\\&&\; vanishes\; at\; all\; boundary\; vertices,\;and \int_F \frac{\partial v}{\partial \nu_F}\,\mathrm{d}s\;is\;continuous\;at\;internal\;faces\;\\&& F\;and\;vanishes\;at\;boundary\;faces\; of \;\mathcal{T}_h\}.
\end{eqnarray*}
The finite element approximation of Problem (\ref{eq1}) reads: Find $\lambda_h \in \mathbb{R}$ and $w_h \in V_h$, such that
\begin{equation}
(\nabla_h^2 w_h,\nabla_h^2 v_h)_{L^2(\Omega)}=\lambda_h(w_h,v_h)_{L^2(\Omega)},\; \text{for any}\; v_h \in V_h. \label{eq46}
\end{equation}
We recall that the operator $\nabla_h^2$ is the discrete counterpart of $\nabla^2$, which is defined element by element since the discrete space $V_h$ is nonconforming.\\
\subsection{Lower bounds of eigenvalues by the three-dimensional rectangular Morley element}
In this section, we show that the approximate eigenvalues are smaller than the exact ones in three-dimensional case.
%\textcolor{green}{By the error estimates of the interpolation $\Pi_h u$ and the finite element solution $u_h$ of the eigenfunction $u$, it follows from \cite{ShiWang} and \cite{HuHuangLin,HuShi1} that
%\begin{eqnarray}
%||u-u_h||_{L^2(\Omega)}\leq Ch^2,\quad |u-u_h|_h^2\leq C h^2, \label{eq49}
%\end{eqnarray}
%Using the triangle inequality and (\ref{eq5}), (\ref{eq49}), it follows that
%\begin{equation}
%||\Pi_h u-u_h||_{L^2(\Omega)}\leq||u-u_h||_{L^2(\Omega)}+||u-\Pi_h u||_{L^2(\Omega)} \leq C h^2. \label{eq50}
%\end{equation}}

\begin{thm}
Let $(\lambda,u)$ and $(\lambda_h,u_h)$ be the solutions to (\ref{eq1}) and (\ref{eq46}), respectively, and assume that $u \in H_0^2(\Omega)\bigcap H^4(\Omega)$,\;then,
\[\lambda_h \leq \lambda,\]
provided that $h$ is small enough.
\end{thm}
\begin{proof}
We use the identity for the errors of the eigenvalue from \cite{ArmDuran,ZhangYangChen}, see also \cite{HuHuangLin}.
\begin{eqnarray}
\lambda-\lambda_h &=&|u-u_h|_h^2-\lambda_h(\Pi_h u-u_h,\Pi_h u-u_h)_{L^2(\Omega)}\nonumber\\&&
                    +\lambda_h\big(||\Pi_h u||_{L^2(\Omega)}^2-||u||_{L^2(\Omega)}^2\big)+2a_h(u-\Pi_h u,u_h),\label{eq15}
\end{eqnarray}
We can bound the second term by (\ref{eq51})
\begin{equation} (\Pi_h u-u_h,\Pi_h u-u_h)_{L^2(\Omega)}\leq C h^4\label{eq52} \end{equation}
and the third term as
\begin{eqnarray}
|||\Pi_h u||_{L^2(\Omega)}^2-||u||_{L^2(\Omega)}^2|&=&|(\Pi_h u,\Pi_h u)_{L^2(\Omega)}-(u,u)_{L^2(\Omega)}|\nonumber\\
                                               &=&|(\Pi_h u-u,\Pi_h u)_{L^2(\Omega)}+(u,\Pi_h u-u)_{L^2(\Omega)}|\nonumber\\
                                               &\leq&Ch^3|u|_{H^3(\Omega)}(||\Pi_h u||_{L^2(\Omega)}+||u||_{L^2(\Omega)})\nonumber\\
                                               &\leq&Ch^3|u|_{H^3(\Omega)}(||\Pi_h u-u||_{L^2(\Omega)}+2||u||_{L^2(\Omega)})\nonumber\\
                                               &\leq&Ch^3|u|_{H^3(\Omega)}^2.\label{eq53}
\end{eqnarray}
Since the fourth term of (\ref{eq15}) is negative and of order $O(h^2)$, we provide a new method to deal with it.
A combination of  the first term and the fourth term of (\ref{eq15}) allows for the following decomposition, like (\ref{eq80})
%\begin{eqnarray}
%|u-u_h|^2_h+2a_h(u-\Pi_h u,u_h)&=&
%(\nabla_h^2(u-\Pi_h u),\nabla_h^2(u-\Pi_h u))_{L^2(\Omega)}\nonumber\\&&+(\nabla_h^2(\Pi_h u-u_h),\nabla_h^2(\Pi_h u-u_h))_{L^2(\Omega)}\nonumber\\&&+2(\nabla_h^2(u-\Pi_h u),\nabla_h^2(\Pi_h u-u_h)_{L^2(\Omega)}\nonumber\\&&+2(\nabla_h^2(u-\Pi_h u),\nabla_h^2u_h)_{L^2(\Omega)}\nonumber\\&=&||\nabla_h^2(u-\Pi_h u)||^2_{L^2(\Omega)}+||\nabla_h^2(\Pi_h u-u_h)||^2_{L^2(\Omega)}\nonumber\\&&+2(\nabla_h^2(u-\Pi_h u),\nabla_h^2\Pi_h u)_{L^2(\Omega)}.\label{eq19}
%\end{eqnarray}
\begin{equation}
\begin{split}
|u-u_h|^2_h+2a_h(u-\Pi_h u,u_h)&=||\nabla_h^2(u-\Pi_h u)||^2_{L^2(\Omega)}+||\nabla_h^2(\Pi_h u-u_h)||^2_{L^2(\Omega)}\\&\quad+2(\nabla_h^2(u-\Pi_h u),\nabla_h^2\Pi_h u)_{L^2(\Omega)}.\label{eq19}
\end{split}
\end{equation}
The first term of (\ref{eq19}) can be further expressed as
%\begin{eqnarray*}
%||\nabla_h^2(u-\Pi_h u)||^2_{L^2(\Omega)}&=&\int_\Omega\bigg[\bigg(\frac{\partial^2(u-\Pi_h u)}{\partial x_1^2}\bigg)^2+\bigg(\frac{\partial^2(u-\Pi_h u)}{\partial x_2^2}\bigg)^2+\bigg(\frac{\partial^2(u-\Pi_h u)}{\partial x_3^2}\bigg)^2\bigg]\\&&\,\mathrm{d}x_1\mathrm{d}x_2\mathrm{d}x_3+2\int_\Omega\bigg[\bigg(\frac{\partial^2(u-\Pi_h u)}{\partial x_1\partial x_2}\bigg)^2+\bigg(\frac{\partial^2(u-\Pi_h u)}{\partial x_1\partial x_3}\bigg)^2\\&&+\bigg(\frac{\partial^2(u-\Pi_h u)}{\partial x_2\partial x_3}\bigg)^2\bigg]\,\mathrm{d}x_1\mathrm{d}x_2\mathrm{d}x_3\\&=&J_1+J_2.
%\end{eqnarray*}
\begin{eqnarray*}
||\nabla_h^2(u-\Pi_h u)||^2_{L^2(\Omega)}&=&\bigg[\int_\Omega\bigg(\frac{\partial^2(u-\Pi_h u)}{\partial x_1^2}\bigg)^2\,\mathrm{d}x_1\mathrm{d}x_2\mathrm{d}x_3+\int_\Omega\bigg(\frac{\partial^2(u-\Pi_h u)}{\partial x_2^2}\bigg)^2\,\mathrm{d}x_1\mathrm{d}x_2\mathrm{d}x_3\\&&+\int_\Omega\bigg(\frac{\partial^2(u-\Pi_h u)}{\partial x_3^2}\bigg)^2\,\mathrm{d}x_1\mathrm{d}x_2\mathrm{d}x_3\bigg]+2\bigg[\int_\Omega\bigg(\frac{\partial^2(u-\Pi_h u)}{\partial x_1\partial x_2}\bigg)^2\,\mathrm{d}x_1\mathrm{d}x_2\mathrm{d}x_3\\&&+\int_\Omega\bigg(\frac{\partial^2(u-\Pi_h u)}{\partial x_1\partial x_3}\bigg)^2\,\mathrm{d}x_1\mathrm{d}x_2\mathrm{d}x_3+\int_\Omega\bigg(\frac{\partial^2(u-\Pi_h u)}{\partial x_2\partial x_3}\bigg)^2\,\mathrm{d}x_1\mathrm{d}x_2\mathrm{d}x_3\bigg]\\&=&J_1+J_2.
\end{eqnarray*}
By definition of $P_T$, we have
\begin{equation}
\frac{\partial^3 \Pi_h u|_K}{\partial x_i^2\partial x_j}=0, \;\text{for any}\; K \in \mathcal{T}_h\; and\; \Pi_h u\in V_h,\;\; i,j=1,2,3,\;i\neq j. \label{eq39}
\end{equation}
Based on this fact, we can show that
\begin{equation}\label{eq58} \beta h^2\leq J_1\;\;  with \;the\; constant\;\beta>0.\end{equation}
See Lemma \ref{lem3} below for a detailed proof of (\ref{eq58}). To prove the final result, we denote $J_3:=2(\nabla_h^2(u-\Pi_h u),\nabla_h^2\Pi_h u)_{L^2(\Omega)}$. It will be proved in Lemma \ref{lem6} below that
\begin{equation}
J_2+J_3\geq 0+Ch^3. \label{eq54}
\end{equation}
By means of  $J_1$, $J_2$, $J_3$, (\ref{eq52})$-$(\ref{eq54}), it follows that (\ref{eq19}) is non-negative and of order $O(h^2)$. Therefore, the sign of $\lambda-\lambda_h$ is non-negative.
\end{proof}
\begin{lem}\label{lem3}
There holds that
\begin{equation*}
\beta h^2\leq J_1\;\;  with \;the\; constant\;\beta>0.
\end{equation*}
\end{lem}
\begin{proof}
Given $K\in\mathcal{T}_h$. Let $P_K$ be defined as in (\ref{eq7}). By (\ref{eq39}), it follows from the triangle inequality and the piecewise inverse estimate that
\begin{eqnarray*}
\sum_{i\neq j=1}^3||\frac{\partial^3u}{\partial x_i^2\partial x_j}||^2_{L^2(\Omega)}&=&\sum_{i\neq j=1}^3\sum_{K\in\mathcal{T}_h}||\frac{\partial^3(u-\Pi_h u)}{\partial x_i^2\partial x_j}||^2_{L^2(K)}\\&\leq &2\sum_{i\neq j=1}^3\sum_{K\in\mathcal{T}_h}\bigg(||\frac{\partial^3(u-P_K u)}{\partial x_i^2\partial x_j}||^2_{L^2(K)}+||\frac{\partial^3(P_K u-\Pi_h u)}{\partial x_i^2\partial x_j}||^2_{L^2(K)}\bigg)\\&\leq C &\sum_{K\in\mathcal{T}_h}||\nabla_h^3(u-P_K u)||^2_{L^2(K)}+h^{-2}\sum_{i=1}^3\sum_{K\in\mathcal{T}_h}||\frac{\partial^2(P_K u-\Pi_h u)}{\partial x_i^2}||^2_{L^2(K)}.
\end{eqnarray*}
Using the triangle inequality and the error estimate (\ref{eq34}) leads to
\begin{equation*}
\sum_{i\neq j=1}^3||\frac{\partial^3u}{\partial x_i^2\partial x_j}||^2_{L^2(\Omega)}\leq C \sum_{K\in\mathcal{T}_h}||\nabla_h^3(u-P_K u)||^2_{L^2(K)}+h^{-2}\sum_{i=1}^3\sum_{K\in\mathcal{T}_h}||\frac{\partial^2(u-\Pi_h u)}{\partial x_i^2}||^2_{L^2(K)}.
\end{equation*}
Using the Poincare inequality and the definition of $P_K$ yields
\begin{equation*}
\sum_{K\in\mathcal{T}_h}||\nabla_h^3(u-P_K u)||^2_{L^2(K)}\leq Ch^2||(I-\Pi_0)\nabla_h^4 w||_{L^2(\Omega)}.
\end{equation*}
Since the piecewise constant functions are dense in the space $L^2(\Omega)$,
\begin{equation*}
||(I-\Pi_0)\nabla_h^4 w||_{L^2(\Omega)}\rightarrow\; 0\;\;when\;\;h\rightarrow \;0.
\end{equation*}
Finally, it follows that
\begin{equation}
h^2\sum_{i\neq j=1}^3||\frac{\partial^3u}{\partial x_i^2\partial x_j}||^2_{L^2(\Omega)}\leq C \sum_{i=1}^3\sum_{K\in\mathcal{T}_h}||\frac{\partial^2(u-\Pi_h u)}{\partial x_i^2}||^2_{L^2(K)}.\label{eq59}
\end{equation}
Next, we prove that $\sum_{i\neq j=1}^3||\frac{\partial^3u}{\partial x_i^2\partial x_j}||^2\neq 0$. In fact, since $u\in H_0^2(\Omega)$, if $||\frac{\partial^3 u}{\partial x_i^2\partial x_j}||_{L^2(\Omega)}=0,\;(i\neq j=1,2,3)$, $u$ should be of the form
\begin{equation*}
u(x_1,x_2,x_3)=\sum_{i=1}^3e_i f_i(x_i)+ax_1x_2x_3+bx_1x_2+cx_1x_3+dx_2x_3,
\end{equation*}
for some functions $f_i(x_i)$ with respect to variable $x_i$, and some constants $e_i, a, b, c, d$. The boundary condition concludes $u$ and its normal derivative vanish on the boundary of $\Omega$. This implies $u\equiv 0$, which contradicts with $u\neq 0$.
\end{proof}
%\begin{rem}
%For the rectangular domain $\Omega$ under consideration, the result (\ref{eq39}) holds provided that $u\in H_0^2(\Omega)$. In fact, if $||\frac{\partial^3 u}{\partial x_i^2\partial x_j}||_{L^2(\Omega)}=0,(i\neq j=1,2,3)$, $u$ is of form
%\begin{equation*}
%u(x_1,x_2,x_3)=\sum_{i=1}^3e_i f_i(x_i)+ax_1x_2x_3+bx_1x_2+cx_1x_3+dx_2x_3,
%\end{equation*}
%for some functions $f_i(x_i)$ with respect to $x_i$, and some constants $e_i, a, b, c, d$. Then the boundary condition concludes the first order partial derivatives of $u$ vanish on the boundary of $\Omega$. Hence the boundary condition indicates $u\equiv 0$, which contradict with $u\neq 0$.
%\end{rem}
\subsection{A refined property of the interpolation operator}
Given $K\in \mathcal{T}_h$, for the sake of simplicity , we  define eighteen bubble functions with respect to the degrees of freedom defined as in (\ref{eq55}) as follows
 \begin{eqnarray}
 \begin{array}{lll}
 \tilde{\varphi}_{i,j}=\xi_i^2\xi_j-\frac{4}{3}\xi_j+\frac{\xi_j^3}{3},\quad i,j=1,2,3,\;i\neq j,\\
 \tilde{\psi}_i=(\xi_i^2-1)^2,\quad i=1,2,3,\\
 \tilde{p}_{i,j}=\xi_i^2+\xi_j^2-\frac{\xi_i^3}{3}-\frac{\xi_j^3}{3}-\frac{1}{3},\quad i=1,2,\;j=2,3,\;i\neq j,\\
 \tilde{q}_{i,j}=\xi_i^3\xi_j-\xi_i\xi_j,\quad i,j=1,2,3,\,i\neq j. \label{eq44}
 \end{array}
 \end{eqnarray}
In fact, it can be checked that
 \begin{eqnarray*}
 &&\left\{
 \begin{array}{lll}
 \tilde{\varphi}_{i,j}(a_k)=0,\quad i,j=1,2,3,\;i\neq j,\,k=1,\ldots,8,\\
 \int_{F_l} \frac{\partial\tilde{\varphi}_{i,j}}{\partial\nu_{F_l}}\;\mathrm{d}s=0,\quad i,j=1,2,3,\;i\neq j,\,l=1,\ldots,6,
  \end{array}
 \right.\\
 &&\left\{
 \begin{array}{lll}
\tilde{\psi}_i(a_k)=0,\quad i=1,2,3,\,k=1,\ldots,8,\\
 \int_{F_l} \frac{\partial\tilde{\psi}_i}{\partial\nu_{F_l}}\;\mathrm{d}s=0,\quad i=1,2,3,\,l=1,\ldots,6,
  \end{array}
 \right.\\
 &&\left\{
 \begin{array}{lll}
 \tilde{p}_{i,j}(a_k)=0,\quad i=1,2,\;j=2,3,\;i\neq j,\;k=1,\ldots,8,\\
 \int_{F_l} \frac{\partial\tilde{p}_{i,j}}{\partial\nu_{F_l}}\;\mathrm{d}s=0,\quad i=1,2,\;j=2.3,\;i\neq j,\,l=1,\ldots,6,
  \end{array}
 \right.\\
 &&\left\{
 \begin{array}{lll}
 \tilde{q}_{i,j}(a_k)=0,\quad i,j=1,2,3,\;i\neq j,\,k=1,\ldots,8,\\
 \int_{F_l} \frac{\partial\tilde{q}_{i,j}}{\partial\nu_{F_l}}\;\mathrm{d}s=0,\quad i,j=1,2,3,\;i\neq j,\,l=1,\ldots,6,
 \end{array}
 \right.
 \end{eqnarray*}
 where $a_k$ are vertices of $K$, and $F_l$ are faces of $K$ .

 In the next lemma, we follow the idea of \cite{HuShi1} to analyze a new refined property for the interpolation operator, which is a basis for the analysis of the term $a_h(u-\Pi_h u,u_h)$.
\begin{lem}\label{lem4}
Given $K\in \mathcal{T}_h$, for any $u\in P_4(K)$ and $ v\in P_T(K)$, there holds that
\begin{eqnarray*}
(\nabla^2(u-\Pi_Ku),\nabla^2v)_{L^2(K)}&=&\frac{h^2}{3}\sum_{i\neq j=1}^3\int_K\frac{\partial^3u}{\partial x_i\partial x_j^2}\frac{\partial^3v}{\partial x_i^3}\,\mathrm{d}x_1\mathrm{d}x_2\mathrm{d}x_3.
\end{eqnarray*}
\end{lem}
\begin{proof}
Let $\xi_1$, $\xi_2 $ and $\xi_3$ be defined as in $(\ref{eq2})$. It follows from the definition of $P_T(K)$ that\\
\begin{eqnarray}
\frac{\partial^2v}{\partial x_i^2}&=&\overline{\frac{\partial^2v}{\partial x_i^2}}+h\frac{\partial^3v}{\partial x_i^3}\xi_i,\quad i=1,2,3,\nonumber\\
\frac{\partial^2v}{\partial x_i\partial x_j}&=&\overline{\frac{\partial^2v}{\partial x_i\partial x_j}}+h\frac{\partial^3v}{\partial x_i\partial x_j\partial x_k}\xi_k,\quad i=1,2,\;j=2,3,\;k=1,2,3,\;i\neq j\neq k.\nonumber\\ \label{eq28}
%\frac{\partial^2v}{\partial x\partial y}&=&\overline{\frac{\partial^2v}{\partial x\partial y}}+h_z\frac{\partial^3v}{\partial x\partial y\partial z}\xi_3,\\
%\frac{\partial^2v}{\partial x\partial z}&=&\overline{\frac{\partial^2v}{\partial x\partial z}}+h\frac{\partial^3v}{\partial x\partial y\partial z}\xi_2,\\
%\frac{\partial^2v}{\partial y\partial z}&=&\overline{\frac{\partial^2v}{\partial y\partial z}}+h\frac{\partial^3v}{\partial x\partial y\partial z}\xi_1.\\
\end{eqnarray}
%where $\overline{f}$ denote the integral average of $f$ over $K$, namely, $\overline{f}=\frac{1}{|K|}\int_K f\mathrm{d}x_1\mathrm{d}x_2\mathrm{d}x_3$.
Since $u\in P_4(K)$, the Taylor expansion and the definition of the operator $\Pi_K$ yield\\
\begin{eqnarray}
u-\Pi_K u&=&\frac{h^3}{2!}\sum_{i\neq j=1}^3\overline{\frac{\partial^3u}{\partial x_i^2\partial x_j}}\tilde{\varphi}_{i,j}+
          \frac{h^4}{4!}\sum_{i=1}^3\frac{\partial^4u}{\partial x_i^4}\tilde{\psi}_i\nonumber\\&&+
          \frac{h^4}{2!2!}\sum_{\substack{i=1,2\\j=2,3\\i\neq j}}\frac{\partial^4u}{\partial x_i^2\partial x_j^2}\tilde{p}_{i,j}+
          \frac{h^4}{3!}\sum_{i\neq j=1}^3\frac{\partial^4u}{\partial x_i^3\partial x_j}\tilde{q}_{i,j},\label{eq10}
\end{eqnarray}
where, $\tilde{\varphi}_{i,j}$, $\tilde{\psi}_i$, $\tilde{p}_{i,j}$, and $\tilde{q}_{i,j}$ are defined as in (\ref{eq44}).
Hence, the second order partial derivative of $u-\Pi_K u$ with respect to the variable $x_i$ reads
\begin{eqnarray}
\frac{\partial^2(u-\Pi_Ku)}{\partial x_i^2}&=&\frac{h}{2!}\sum_{i\neq j=1}^3\bigg(\overline{\frac{\partial^3u}{\partial x_i^2\partial x_j}}2\xi_j+\overline{\frac{\partial^3u}{\partial x_i\partial x_j^2}}2\xi_i\bigg)+  \frac{h^2}{4!}\frac{\partial^4u}{\partial x_i^4}(12\xi_i^2-4)\nonumber\\&&+
                                          \frac{h^2}{2!2!}\sum_{i\neq j=1}^3\frac{\partial^4u}{\partial x_i^2\partial x_j^2}(2\xi_j^2-\frac{2}{3})+
                                          \frac{h^2}{3!}\sum_{i\neq j=2}^3\frac{\partial^4u}{\partial x_i^3\partial x_j}6\xi_i\xi_j,\nonumber\\ \label{eq26}
\end{eqnarray}
%A combination of $(\ref{eq10})$ and $(\ref{eq26})$ plus some elementary calculation give
%\begin{eqnarray*}
%&&\int_K\frac{\partial^2(u-\Pi_Ku)}{\partial x_1^2}\frac{\partial^2v}{\partial x_1^2}\,\mathrm{d}x_1\mathrm{d}x_2\mathrm{d}x_3\\&=&\int_K \bigg(\frac{h}{2!}\sum_{j=2}^3(\overline{\frac{\partial^3u}{\partial x_1^2\partial x_j}}2\xi_j+\overline{\frac{\partial^3u}{\partial x_1\partial x_j^2}}2\xi_1)+
%                                          \frac{h^2}{2!2!}\sum_{j=2}^3\frac{\partial^4u}{\partial x_1^2\partial x_j^2}(2\xi_j^2-\frac{2}{3})\\&&+\frac{h^2}{4!}\frac{\partial^4u}{\partial x^4}(12\xi_1^2-4)+
%                                          \frac{h^2}{3!}\sum_{j=2}^3\frac{\partial^4u}{\partial x_1^3\partial x_j}6\xi_1\xi_j\bigg)\Big(\overline{\frac{\partial^2v}{\partial x_1^2}}+h\frac{\partial^3v}{\partial x_1^3}\xi_1\Big)\;\mathrm{d}x_1\mathrm{d}x_2\mathrm{d}x_3.
%\end{eqnarray*}
%Using parity of functions and symmetry of domains, we can get that
A similar argument for the two-dimensional case shows
\begin{eqnarray*}
\int_K\frac{\partial^2(u-\Pi_Ku)}{\partial x_i^2}\frac{\partial^2v}{\partial x_i^2}\,\mathrm{d}x_1\mathrm{d}x_2\mathrm{d}x_3&=&
                 \frac{h^2}{3}\sum_{i\neq j=1}^3\int_K\frac{\partial^3u}{\partial x_i\partial x_j^2}\frac{\partial^3v}{\partial x_i^3}\,\mathrm{d}x_1\mathrm{d}x_2\mathrm{d}x_3.
\end{eqnarray*}
%A similar argument proves
%\begin{eqnarray*}
%\int_K\frac{\partial^2(u-\Pi_Ku)}{\partial x_2^2}\frac{\partial^2v}{\partial x_2^2}\,\mathrm{d}x_1\mathrm{d}x_2\mathrm{d}x_3&=& \frac{h^2}{3}\sum_{j=1,3}\int_K\frac{\partial^3u}{\partial x_2\partial x_j^2}\frac{\partial^3v}{\partial x_2^3}\,\mathrm{d}x_1\mathrm{d}x_2\mathrm{d}x_3,\\
%\int_K\frac{\partial^2(u-\Pi_Ku)}{\partial x_3^2}\frac{\partial^2v}{\partial x_3^2}\,\mathrm{d}x_1\mathrm{d}x_2\mathrm{d}x_3&=& \frac{h^2}{3}\sum_{j=1}^2\int_K\frac{\partial^3u}{\partial x_3\partial x_j^2}\frac{\partial^3v}{\partial x_3^3}\,\mathrm{d}x_1\mathrm{d}x_2\mathrm{d}x_3.
%\end{eqnarray*}
The second order mixed partial derivative of $u-\Pi_K u$ is
\begin{eqnarray*}
\frac{\partial^2(u-\Pi_Ku)}{\partial x_i\partial x_j}&=&\frac{h}{2!}\overline{\frac{\partial^3u}{\partial x_i^2\partial x_j}}2\xi_i+
                                                   \frac{h}{2!}\overline{\frac{\partial^3u}{\partial x_i\partial x_j^2}}2\xi_j+
                                                   \frac{h^2}{2!2!}\frac{\partial^4u}{\partial x_i^2\partial x_j^2}4\xi_i\xi_j+
                                                  \\&& \frac{h^2}{3!}\frac{\partial^4u}{\partial x_i^3\partial x_j}(3\xi_i^2-1)+
                                                   \frac{h^2}{3!}\frac{\partial^4u}{\partial x_i\partial x_j^3}(3\xi_j^2-1),\\&& i=1,2,\,j=2,3,\,i\neq j.\label{eq11}
\end{eqnarray*}
This and (\ref{eq28}) lead to
\begin{eqnarray}
\int_K\frac{\partial^2(u-\Pi_Ku)}{\partial x_i\partial x_j}\frac{\partial^2v}{\partial x_i\partial x_j}\,\mathrm{d}x_1\mathrm{d}x_2\mathrm{d}x_3&=&0,\quad i=1,2,\,j=2,3,\,i\neq j,\nonumber\\\label{eq12}
\end{eqnarray}
which completes the proof.
\end{proof}
Next, we use Lemma \ref{lem4} to analyze the key term in the proof of Theorem 3.1.
\begin{lem}\label{lem5}
Suppose that $w \in H_0^2(\Omega) \bigcap H^4(\Omega)$ with $\Omega\subset \mathbb{R}^3$. Then,
\begin{eqnarray*}
(\nabla_h^2(w-\Pi_h w),\nabla_h^2 \Pi_h w)_{L^2(\Omega)}&=&
-\frac{h^2}{3}\sum_{i\neq j\neq k=1}^3\sum_{K\in\mathcal{T}_h}\int_K\frac{\partial^3w}{\partial x_i\partial x_j^2}\frac{\partial^3w}{\partial x_i\partial x_k^2}\,\mathrm{d}x_1\mathrm{d}x_2\mathrm{d}x_3\\&&+Ch^3.
\end{eqnarray*}
%\begin{eqnarray}
%(\nabla_h^2(w-\Pi_h w),\nabla_h^2 \Pi_h w)_{L^2(\Omega)}&\geq&
%-\frac{2h^2}{3}\sum_{K\in\mathcal{T}_h}\int_K\frac{\partial^3w}{\partial x_1\partial x_2^2}\frac{\partial^3w}{\partial x_1\partial x_3^2}\,\mathrm{d}x_1\mathrm{d}x_2\mathrm{d}x_3\nonumber\\&&
%-\frac{2h^2}{3}\sum_{K\in\mathcal{T}_h}\int_K\frac{\partial^3w}{\partial x_1^2\partial x_2}\frac{\partial^3w}{\partial x_2\partial x_3^2}\,\mathrm{d}x_1\mathrm{d}x_2\mathrm{d}x_3\nonumber\\&&
%-\frac{2h^2}{3}\sum_{K\in\mathcal{T}_h}\int_K\frac{\partial^3w}{\partial x_1^2\partial x_3}\frac{\partial^3w}{\partial x_2^2\partial x_3}\,\mathrm{d}x_1\mathrm{d}x_2\mathrm{d}x_3\nonumber\\&&
%+Ch^3.
%\end{eqnarray}
%\begin{eqnarray}
%(\nabla_h^2(w-\Pi_h w),\nabla_h^2 \Pi_h w)_{L^2(\Omega)}&\geq&
%-\frac{2h^2}{3}\sum_{\substack{i=1,2,3\\j=1,2\\k=2,3\\i\neq j\neq k}}\sum_{K\in\mathcal{T}_h}\int_K\frac{\partial^3w}{\partial x_1\partial x_2^2}\frac{\partial^3w}{\partial x_1\partial x_3^2}\,\mathrm{d}x_1\mathrm{d}x_2\mathrm{d}x_3\nonumber\\&&
%-\frac{2h^2}{3}\sum_{K\in\mathcal{T}_h}\int_K\frac{\partial^3w}{\partial x_1^2\partial x_2}\frac{\partial^3w}{\partial x_2\partial x_3^2}\,\mathrm{d}x_1\mathrm{d}x_2\mathrm{d}x_3\nonumber\\&&
%-\frac{2h^2}{3}\sum_{K\in\mathcal{T}_h}\int_K\frac{\partial^3w}{\partial x_1^2\partial x_3}\frac{\partial^3w}{\partial x_2^2\partial x_3}\,\mathrm{d}x_1\mathrm{d}x_2\mathrm{d}x_3\nonumber\\&&
%+Ch^3.
%\end{eqnarray}
\end{lem}
\begin{proof}
Given $K \in \mathcal{T}_h$, let the interpolation operator $P_K$ be defined as in (\ref{eq7}), which leads to the following decomposition
\begin{eqnarray}
(\nabla_h^2(w-\Pi_h w),\nabla_h^2\Pi_h w)_{L^2(\Omega)}&=&\sum_{K\in\mathcal{T}_h} (\nabla_h^2(P_Kw-\Pi_K P_K w),\nabla_h^2\Pi_K w)_{L^2(K)}\nonumber\\&&
                                                         +\sum_{K\in\mathcal{T}_h} (\nabla_h^2(I-\Pi_K)(I-P_K)w,\nabla_h^2\Pi_K w)_{L^2(K)}\nonumber\\
                                                       &=&I_1+I_2.  \label{eq20}
\end{eqnarray}
Let $u=P_Kw$ and $v=\Pi_Kw$ in Lemma \ref{lem4}, the first term $I_1$ on the right-hand side of  (\ref{eq20}) can be rewritten as

\begin{eqnarray*}
%I_1&=&\sum_{K\in\mathcal{T}_h} \frac{h^2}{3}\int_K\frac{\partial^3P_Kw}{\partial x_1\partial x_2^2}\frac{\partial^3\Pi_Kw}{\partial x_1^3}\,\mathrm{d}x_1\mathrm{d}x_2\mathrm{d}x_3\\&&+\sum_{K\in\mathcal{T}_h} \frac{h^2}{3}\int_K\frac{\partial^3P_Kw}{\partial x_1\partial x_3^2}\frac{\partial^3\Pi_Kw}{\partial x_1^3}\,\mathrm{d}x_1\mathrm{d}x_2\mathrm{d}x_3\\&&+\sum_{K\in\mathcal{T}_h} \frac{h^2}{3}\int_K\frac{\partial^3P_Kw}{\partial x_1^2\partial x_2}\frac{\partial^3\Pi_Kw}{\partial x_2^3}\,\mathrm{d}x_1\mathrm{d}x_2\mathrm{d}x_3\\&&+\sum_{K\in\mathcal{T}_h} \frac{h^2}{3}\int_K\frac{\partial^3P_Kw}{\partial x_2\partial x_3^2}\frac{\partial^3\Pi_Kw}{\partial x_2^3}\,\mathrm{d}x_1\mathrm{d}x_2\mathrm{d}x_3\\&&+\sum_{K\in\mathcal{T}_h} \frac{h^2}{3}\int_K\frac{\partial^3P_Kw}{\partial x_1^2\partial x_3}\frac{\partial^3\Pi_Kw}{\partial x_3^3}\,\mathrm{d}x_1\mathrm{d}x_2\mathrm{d}x_3\\&&+\sum_{K\in\mathcal{T}_h} \frac{h^2}{3}\int_K\frac{\partial^3P_Kw}{\partial x_2^2\partial x_3}\frac{\partial^3\Pi_Kw}{\partial x_3^3}\,\mathrm{d}x_1\mathrm{d}x_2\mathrm{d}x_3\\
%&=&I^{(1)}_1+I^{(2)}_1+I^{(3)}_1+I^{(4)}_1+I^{(5)}_1+I^{(6)}_1.
I_1=\frac{h^2}{3}\sum_{i\neq j=1}^3\sum_{K\in\mathcal{T}_h} \int_K\frac{\partial^3P_Kw}{\partial x_i\partial x_j^2}\frac{\partial^3\Pi_Kw}{\partial x_i^3}\,\mathrm{d}x_1\mathrm{d}x_2\mathrm{d}x_3=\sum_{i\neq j=1}^3I_1^{(i,j)}.
\end{eqnarray*}
The term $I_1^{(i,j)}$ has the following decomposition
\begin{eqnarray*}
I_1^{(i,j)}&=&\sum_{K\in\mathcal{T}_h} \frac{h^2}{3}\int_K\frac{\partial^3w}{\partial x_i\partial x_j^2}\frac{\partial^3w}{\partial x_i^3}\,\mathrm{d}x_1\mathrm{d}x_2\mathrm{d}x_3\\&&+
\sum_{K\in\mathcal{T}_h} \frac{h^2}{3}\int_K\frac{\partial^3w}{\partial x_i\partial x_j^2}\frac{\partial^3(\Pi_Kw-w)}{\partial x_i^3}\,\mathrm{d}x_1\mathrm{d}x_2\mathrm{d}x_3\\&&+
\sum_{K\in\mathcal{T}_h} \frac{h^2}{3}\int_K\frac{\partial^3(P_K-I)w}{\partial x_i\partial x_j^2}\frac{\partial^3\Pi_Kw}{\partial x_i^3}\,\mathrm{d}x_1\mathrm{d}x_2\mathrm{d}x_3\\&=&I_{1,1}^{(i,j)}+I_{1,2}^{(i,j)}+I_{1,3}^{(i,j)}.
\end{eqnarray*}
After integrating by parts twice, the first term of $I_1^{(i,j)}$ can be expressed as
\begin{eqnarray*}
I_{1,1}^{(i,j)}=\sum_{K\in\mathcal{T}_h} \frac{h^2}{3}\int_K\frac{\partial^3w}{\partial x_i\partial x_j^2}\frac{\partial^3w}{\partial x_i^3}\,\mathrm{d}x_1\mathrm{d}x_2\mathrm{d}x_3=\frac{h^2}{3}||\frac{\partial^3w}{\partial x_i^2\partial x_j}||^2_{L^2(\Omega)}.
\end{eqnarray*}
Since $||\frac{\partial^3 \Pi_K w}{\partial x_i^3}||_{L^2(K)}\leq C|w|_{H^3(K)}$ is bounded, it follows that
\begin{eqnarray*}
I_{1,3}^{(i,j)}=\sum_{K\in\mathcal{T}_h} \frac{h^2}{3}\int_K\frac{\partial^3(P_K-I)w}{\partial x_i\partial x_j^2}\frac{\partial^3\Pi_Kw}{\partial x_i^3}\,\mathrm{d}x_1\mathrm{d}x_2\mathrm{d}x_3\leq Ch^3|w|_{H^4(\Omega)}^2.
\end{eqnarray*}
Due to the Taylor expansion, see (\ref{eq10}), up to a higher order term, it follows that
\begin{eqnarray*}
I_{1,2}^{(i,j)}&=& \sum_{K\in\mathcal{T}_h} \frac{h^2}{3}\int_K\frac{\partial^3w}{\partial x_i\partial x_j^2}\frac{\partial^3(\Pi_Kw-w)}{\partial x_i^3}\,\mathrm{d}x_1\mathrm{d}x_2\mathrm{d}x_3\\&=&-\sum_{K\in\mathcal{T}_h} \frac{h^2}{3}\int_K\frac{\partial^3w}{\partial x_i\partial x_j^2}\frac{\partial^3w}{\partial x_i\partial x_j^2}\,\mathrm{d}x_1\mathrm{d}x_2\mathrm{d}x_3-\sum_{K\in\mathcal{T}_h} \frac{h^2}{3}\int_K\frac{\partial^3w}{\partial x_i\partial x_k^2}\frac{\partial^3w}{\partial x_i\partial x_j^2}\,\mathrm{d}x_1\mathrm{d}x_2\mathrm{d}x_3\\&&+Ch^3\\
&=&-\frac{h^2}{3}\sum_{K\in\mathcal{T}_h}\int_K(\frac{\partial^3w}{\partial x_i\partial x_j^2})^2\,\mathrm{d}x_1\mathrm{d}x_2\mathrm{d}x_3-\frac{h^2}{3}\sum_{K\in\mathcal{T}_h}\int_K\frac{\partial^3w}{\partial x_i\partial x_k^2}\frac{\partial^3w}{\partial x_i\partial x_j^2}\,\mathrm{d}x_1\mathrm{d}x_2\mathrm{d}x_3\\&&+Ch^3,\quad\;i\neq j\neq k=1,2,3.
\end{eqnarray*}
A summary of these three terms leads to
\begin{eqnarray*}
I^{(i,j)}_1&=&\frac{h^2}{3}||\frac{\partial^3w}{\partial x_i^2\partial x_j}||^2_{L^2(\Omega)}-\frac{h^2}{3}\sum_{K\in\mathcal{T}_h}\int_K\bigg(\frac{\partial^3w}{\partial x_i\partial x_j^2}\bigg)^2\,\mathrm{d}x_1\mathrm{d}x_2\mathrm{d}x_3\\&&-\frac{h^2}{3}\sum_{K\in\mathcal{T}_h}\int_K\frac{\partial^3w}{\partial x_i\partial x_j^2}\frac{\partial^3w}{\partial x_i\partial x_k^2}\,\mathrm{d}x_1\mathrm{d}x_2\mathrm{d}x_3+Ch^3,\quad\;i\neq j\neq k=1,2,3.
\end{eqnarray*}
It is a consequence of the sum of $I_1^{(i,j)}$ that
\begin{eqnarray}
I_1&=& -\frac{h^2}{3}\sum_{i\neq j\neq k=1}^3\sum_{K\in\mathcal{T}_h}\int_K\frac{\partial^3w}{\partial x_i\partial x_j^2}\frac{\partial^3w}{\partial x_i\partial x_k^2}\,\mathrm{d}x_1\mathrm{d}x_2\mathrm{d}x_3+Ch^3. \label{eq36}
\end{eqnarray}
We turn to the second term $I_2$ on the right-hand side of (\ref{eq20}), which can be estimated by the error estimates of (\ref{eq34}) as
\begin{equation*}
|I_2|\leq Ch\sum_{K\in\mathcal{T}_h}||\nabla_h^3(I-P_K)w||_{L^2(K)}|w|_{H^3(\Omega)}.
\end{equation*}
The definition of the projection operator $P_K$ gives
\begin{equation*}
\int_K \nabla_h^3(I-P_K)w\;\mathrm{d}x_1\mathrm{d}x_2\mathrm{d}x_3=0.
\end{equation*}
By the Poincare inequality, and the commuting property of (\ref{eq8}),
\begin{eqnarray}
|I_2|&\leq & Ch^2\sum_{K\in\mathcal{T}_h}||\nabla_h^4(I-P_K)w||_{L^2(K)}|w|_{H^3(\Omega)}\nonumber\\&\leq&Ch^2||(I-\Pi_0)\nabla_h^4 w||_{L^2(\Omega)}|w|_{H^3(\Omega)}.\label{eq29}
\end{eqnarray}
Since the piecewise constant functions are dense in the space $L^2(\Omega)$,
\begin{equation}
||(I-\Pi_0)\nabla_h^4 w||_{L^2(\Omega)}\rightarrow\; 0\;\;when\;\;h\rightarrow \;0.\label{eq30}
\end{equation}
A summary of (\ref{eq36}), (\ref{eq29}) and (\ref{eq30}) completes the proof.
\end{proof}
\begin{lem}\label{lem6}
Let u be the eigenfunction of Problem (\ref{eq1}). Assume that $u \in H_0^2(\Omega) \bigcap H^4(\Omega)$ with $\Omega\subset \mathbb{R}^3$, then
\begin{equation*}
J_2+J_3\geq 0+Ch^3.
\end{equation*}
\end{lem}
\begin{proof}
 $J_2$ can be expressed as
\begin{eqnarray*}
J_2&=&2\int_\Omega\bigg[\bigg(\frac{\partial^2(u-\Pi_h u)}{\partial x_1\partial x_2}\bigg)^2+\bigg(\frac{\partial^2(u-\Pi_h u)}{\partial x_1\partial x_3}\bigg)^2+\bigg(\frac{\partial^2(u-\Pi_h u)}{\partial x_2\partial x_3}\bigg)^2\bigg]\,\mathrm{d}x_1\mathrm{d}x_2\mathrm{d}x_3\\&=&2\sum_{K\in\mathcal{T}_h}\int_K\bigg[\bigg(\frac{\partial^2(u-\Pi_K u)}{\partial x_1\partial x_2}\bigg)^2+\bigg(\frac{\partial^2(u-\Pi_K u)}{\partial x_1\partial x_3}\bigg)^2+\bigg(\frac{\partial^2(u-\Pi_K u)}{\partial x_2\partial x_3}\bigg)^2\bigg]\,\mathrm{d}x_1\mathrm{d}x_2\mathrm{d}x_3.
\end{eqnarray*}
By the Taylor expansion (\ref{eq10}), up to a higher order term, $u-\Pi_K u$ has the following expression,
\begin{eqnarray*}
u-\Pi_K u&=&\frac{h^3}{2!}\sum_{i\neq j=1}^3\overline{\frac{\partial^3 u}{\partial x_i^2\partial x_j}}\tilde{\varphi}_{i,j}+O(h^4)\\&=&\frac{h^3}{2!}\sum_{i\neq j=1}^3\overline{\frac{\partial^3 u}{\partial x_i^2\partial x_j}}\bigg(\xi_i^2\xi_j-\frac{4}{3}\xi_j+\frac{\xi_j^3}{3}\bigg)+O(h^4),
\end{eqnarray*}
which $\xi_i$ are defined as in (\ref{eq2}). The second order mixed partial derivative of $u-\Pi_K u$ yields
\begin{eqnarray*}
\frac{\partial^2(u-\Pi_Ku)}{\partial x_i\partial x_j}&=&\frac{h}{2!}\overline{\frac{\partial^3u}{\partial x_i^2\partial x_j}}2\xi_i+
                                                   \frac{h}{2!}\overline{\frac{\partial^3u}{\partial x_i\partial x_j^2}}2\xi_j+O(h^2),\quad i=1,2,\;j=2,3,\;i\neq j.
\end{eqnarray*}
Since $\overline{\frac{\partial^3u}{\partial x_i^2\partial x_j}}$ and $\overline{\frac{\partial^3u}{\partial x_i\partial x_j^2}}$ are constants, we can get that by parity of functions and symmetry of domains:
\begin{equation*}
\int_K \overline{\frac{\partial^3u}{\partial x_i^2\partial x_j}}\;\;\overline{\frac{\partial^3u}{\partial x_i\partial x_j^2}}\xi_i\xi_j\;\mathrm{d}x_1\mathrm{d}x_2\mathrm{d}x_3=0,\quad i=1,2,\;j=2,3,\;i\neq j.
\end{equation*}
This yields
\begin{eqnarray*}
\int_K \bigg(\frac{\partial^2(u-\Pi_Ku)}{\partial x_i\partial x_j}\bigg)^2\;\mathrm{d}x_1\mathrm{d}x_2\mathrm{d}x_3&=&h^2\int_K\bigg(\frac{\partial^3u}{\partial x_i^2\partial x_j}\bigg)^2\xi_i^2\;\mathrm{d}x_1\mathrm{d}x_2\mathrm{d}x_3\\&&+h^2\int_K\bigg(\frac{\partial^3u}{\partial x_i\partial x_j^2}\bigg)^2\xi_j^2\;\mathrm{d}x_1\mathrm{d}x_2\mathrm{d}x_3+O(h^4),\\&& i=1,2,\;j=2,3,\;i\neq j.
\end{eqnarray*}
Thus, $J_2$ can be rewritten as
\begin{equation*}
J_2=2h^2\sum_{i\neq j=1}^3\sum_{K\in\mathcal{T}_h}\int_K\bigg(\frac{\partial^3u}{\partial x_i^2\partial x_j}\bigg)^2\xi_i^2\,\mathrm{d}x_1\mathrm{d}x_2\mathrm{d}x_3+O(h^4).
\end{equation*}
%By the Taylor expansion (\ref{eq10}), up to a higher order term, $J_2$ can be rewritten as
%\begin{eqnarray*}
%J_2&=&2h^2\sum_{K\in\mathcal{T}_h}\int_K\bigg[\bigg(\frac{\partial^3u}{\partial x_1^2\partial x_2}\bigg)^2\xi_1^2+\bigg(\frac{\partial^3u}{\partial x_1\partial x_2^2}\bigg)^2\xi_2^2+\bigg(\frac{\partial^3u}{\partial x_1^2\partial x_3}\bigg)^2\xi_1^2\\&&+\bigg(\frac{\partial^3u}{\partial x_1\partial x_3^2}\bigg)^2\xi_3^2+\bigg(\frac{\partial^3u}{\partial x_2^2\partial x_3}\bigg)^2\xi_2^2+\bigg(\frac{\partial^3u}{\partial x_2\partial x_3^2}\bigg)^2\xi_3^2\bigg]\,\mathrm{d}x_1\mathrm{d}x_2\mathrm{d}x_3+O(h^4)\\&=&2h^2\sum_{i\neq j=1}^3\sum_{K\in\mathcal{T}_h}\int_K\bigg(\frac{\partial^3u}{\partial x_i^2\partial x_j}\bigg)^2\xi_i^2\,\mathrm{d}x_1\mathrm{d}x_2\mathrm{d}x_3+O(h^4),
%\end{eqnarray*}
%which $\xi_i$ are defined as in (\ref{eq2}).
Concerning the last term of (\ref{eq19}), in Lemma \ref{lem5}, it follows that
\begin{eqnarray}
(\nabla_h^2(u-\Pi_h u),\nabla_h^2\Pi_h u)_{L^2(\Omega)} &=&-\frac{h^2}{3}\sum_{i\neq j\neq k=1}^3\sum_{K\in\mathcal{T}_h}\int_K\frac{\partial^3u}{\partial x_i\partial x_j^2}\frac{\partial^3u}{\partial x_i\partial x_k^2}\,\mathrm{d}x_1\mathrm{d}x_2\mathrm{d}x_3\nonumber\\&&+Ch^3.  \label{eq14}
\end{eqnarray}
By the Cauchy-Schwartz inequality, this yields
\begin{eqnarray*}
%2(\nabla_h^2(u-\Pi_h u),\nabla_h^2\Pi_h u)_{L^2(\Omega)} &\geq&-\frac{2h^2}{3}\sum_{i\neq j=1}^3\sum_{K\in\mathcal{T}_h}\int_K(\frac{\partial^3u}{\partial x_i\partial x_j^2})^2\,\mathrm{d}x_1\mathrm{d}x_2\mathrm{d}x_3+Ch^3.
J_3&\geq&-\frac{2h^2}{3}\sum_{i\neq j=1}^3\sum_{K\in\mathcal{T}_h}\int_K\bigg(\frac{\partial^3u}{\partial x_i\partial x_j^2}\bigg)^2\,\mathrm{d}x_1\mathrm{d}x_2\mathrm{d}x_3+Ch^3.
\end{eqnarray*}
Hence,
\begin{eqnarray*}
J_2+J_3&\geq& 2h^2\sum_{i\neq j=1}^3\sum_{K\in\mathcal{T}_h}\int_K\bigg(\frac{\partial^3u}{\partial x_i^2\partial x_j}\bigg)^2\xi_i^2\,\mathrm{d}x_1\mathrm{d}x_2\mathrm{d}x_3\\&&-\frac{2h^2}{3}\sum_{i\neq j=1}^3\sum_{K\in\mathcal{T}_h}\int_K\bigg(\frac{\partial^3u}{\partial x_i\partial x_j^2}\bigg)^2\,\mathrm{d}x_1\mathrm{d}x_2\mathrm{d}x_3+Ch^3.
\end{eqnarray*}
Given $K\in\mathcal{T}_h$, it is sufficient to analyze the following term
\begin{eqnarray*}
&&2h^2\int_K\bigg(\frac{\partial^3u}{\partial x_1^2\partial x_2}\bigg)^2\xi_1^2\,\mathrm{d}x_1\mathrm{d}x_2\mathrm{d}x_3-\frac{2h^2}{3}\int_K\bigg(\frac{\partial^3u}{\partial x_1^2\partial x_2}\bigg)^2\,\mathrm{d}x_1\mathrm{d}x_2\mathrm{d}x_3\nonumber\\
&=&2h^2\int_K\bigg(\frac{\partial^3u}{\partial x_1^2\partial x_2}\bigg)^2(\xi_1^2-\frac{1}{3})\,\mathrm{d}x_1\mathrm{d}x_2\mathrm{d}x_3\nonumber\\
&=&2h^2\int_K\bigg[\bigg(\frac{\partial^3u}{\partial x_1^2\partial x_2}\bigg)^2-\bigg(\overline{\frac{\partial^3u}{\partial x_1^2\partial x_2}}\bigg)^2\bigg](\xi_1^2-\frac{1}{3})\,\mathrm{d}x_1\mathrm{d}x_2\mathrm{d}x_3\nonumber\\
&&+2h^2\int_K\bigg(\overline{\frac{\partial^3u}{\partial x_1^2\partial x_2}}\bigg)^2(\xi_1^2-\frac{1}{3})\,\mathrm{d}x_1\mathrm{d}x_2\mathrm{d}x_3\label{eq17}\\
&\leq& Ch^3|u|_{H^4(K)}^2, \label{eq23}
\end{eqnarray*}
where we use the fact that the second term of the second identity vanishes.
\end{proof}

\section{Numerical results}
In this section, we present some numerical results to demonstrate our theoretical results. Herein, we denote $r$ as the rate of convergence. In the first example, we consider the following eigenvalue problem of the two-dimensional biharmonic equation imposed the following boundary conditions
\begin{equation}
\Delta^2 u=\lambda u,\; \text{in}\; \Omega=[0,1]^2,
\end{equation}
(a)The clamped boundary condition
\begin{equation*}
u=\partial_\nu u=0\;\;\text{on}\; \partial\Omega.
\end{equation*}
(b)The simply supported boundary condition
\begin{equation*}
u=0\;\;   \text{on}\; \partial\Omega.
\end{equation*}
We partition the domain $\Omega$ into the uniform squares with the meshsize $h=\frac{1}{N}$ for some integer. The first six eigenvalues are listed in Table 1, Table 2, respectively.

In the second example, we consider the following eigenvalue problem of the three-dimensional biharmonic equation imposed the following boundary conditions
\begin{equation}
\Delta^2 u=\lambda u\; \text{in}\; \Omega=[0,1]^3,
\end{equation}
(c)The clamped boundary condition
\begin{equation*}
u=\partial_\nu u=0\;\;   \text{on}\; \partial\Omega.
\end{equation*}
(d)The simply supported boundary condition
\begin{equation*}
u=0\;\;   \text{on}\; \partial\Omega.
\end{equation*}
We partition the domain $\Omega$ into the uniform cubics with the meshsize $h=\frac{1}{N}$ for some integer.  The first six eigenvalues are listed in Table 3, Table 4, respectively.

%\floatbarrier

 From the tables, we can find that the discrete eigenvalues converge monotonically from below to the exact ones for all the boundary conditions under consideration.
\begin{rem}
In this paper, we provide the proof of lower bounds of eigenvalues for the clamped boundary condition. However, the analysis in this paper does not cover the case for the simply supported boundary condition. From the numerical results, we can find that it also holds for the simply supported boundary condition.
\end{rem}

\newpage

\begin{table}[h]
 \centering
 \caption{The first six eigenvalues for the clamped boundary condition in 2-D case}
 \begin{tabular}{ccccccc}
 \hline
 N              &4          &8           &12        &16        &32   &\\\hline
 $\lambda_{1,h}$&1075.8563  &1223.1076   &1261.1771 &1275.5592 &1289.9935&$\nearrow$\vspace{2mm}\\
 $\lambda_{2,h}$&4481.4554  &5017.6904   & 5205.0626 & 5280.6461 &5359.1648&$\nearrow$\vspace{2mm}\\
 $\lambda_{3,h}$&4481.4554  &5017.6904   & 5205.0626 & 5280.6461 &5359.1648&$\nearrow$\vspace{2mm}\\
 $\lambda_{4,h}$&7697.5590  & 9953.5911   &10819.5084 &11183.7787 &11572.2467&$\nearrow$\vspace{2mm}\\
 $\lambda_{5,h}$&15704.3199  &16244.1142   &16743.9469 &16971.6555 &17222.4239&$\nearrow$\vspace{2mm}\\
 $\lambda_{6,h}$&16296.5202  &16520.5023   &16955.9294 &17162.3431 &17393.3846&$\nearrow$\vspace{2mm}\\\hline
 \end{tabular}
\end{table}

\begin{table}[h]
 \centering
 \caption{The first six eigenvalues for the simply supported boundary condition in 2-D case}
 \begin{tabular}{cccccccc}
 \hline
 N              &4          &8           &12        &16        &32       & & Exact  \\\hline
 $\lambda_{1,h}$&347.5266  &377.6791   & 384.1862 &386.5430 & 388.8563 &$\nearrow$& $ 4\pi^4\approx$389.6364\vspace{2mm}\\
 $r$ &--- &1.816267&1.937758&1.968793&1.987512& &\vspace{2mm}\\
 $\lambda_{2,h}$&2104.3141  &2323.3219   &2382.3420 & 2404.8176 &2427.4598&$\nearrow$&$25\pi^4\approx$2435.2273
 \vspace{2mm} \\
 $\lambda_{3,h}$&2104.3141  &2323.3219   &2382.3420 & 2404.8176 &2427.4598&$\nearrow$&$ 25\pi^4\approx$2435.2273
 \vspace{2mm}\\
  $r$&--- &1.564173&1.848565&1.923527&1.969013& &\vspace{2mm}\\
 $\lambda_{4,h}$&4428.5078  &5560.4260   &5905.5665 &6042.8650 &6184.6886&$\nearrow$&$ 64\pi^4\approx$6234.1818\vspace{2mm}\\
 $r$ &--- &1.422239&1.770756&1.880399&1.950661& &\vspace{2mm}\\
 $\lambda_{5,h}$&8883.3154  &9298.3330   & 9516.2149 &9608.4258 &9706.2378&$\nearrow$&$ 100\pi^4\approx$9740.9091\vspace{2mm}\\
 $\lambda_{6,h}$&8883.3154  &9298.3330   & 9516.2149 &9608.4258 &9706.2378&$\nearrow$&$ 100\pi^4\approx$9740.9091\vspace{2mm}\\
$r$ &--- &0.954369&1.671838&1.836346&1.933997& &\vspace{2mm}\\\hline
 \end{tabular}
\end{table}
\newpage
\begin{table}[h]
 \centering
 \caption{The first six eigenvalues for the clamped boundary condition in 3-D case}
 \begin{tabular}{cccccc}
 \hline
 N              &4          &8           &12        &16          &\\\hline
 $\lambda_{1,h}$&1714.3524  &2136.8429   &2255.9156 &2302.1447 &$\nearrow$\vspace{2mm}\\
 $\lambda_{2,h}$&5174.6283  &6369.4367   &6796.5628 &6972.4742 &$\nearrow$\vspace{2mm}\\
 $\lambda_{3,h}$&5174.6283  &6369.4367   &6796.5628 &6972.4742 &$\nearrow$\vspace{2mm}\\
 $\lambda_{4,h}$&5174.6283  &6369.4367   &6796.5628 &6972.4742 &$\nearrow$\vspace{2mm}\\
 $\lambda_{5,h}$&8539.6777  & 11655.1631   &12920.9204 &13468.3120&$\nearrow$ \vspace{2mm}\\
 $\lambda_{6,h}$&8539.6777  & 11655.1631  &12920.9204 &13468.3120 &$\nearrow$\vspace{2mm}\\\hline
 \end{tabular}
\end{table}

\begin{table}[h]
 \centering
 \caption{The first six eigenvalues for the simply supported boundary condition in 3-D case}
 \begin{tabular}{ccccccc}
 \hline
 N              &4          &8           &12        &16               & & Exact  \\\hline
 $\lambda_{1,h}$&718.3621  &828.0498   &854.1259 & 863.7983       &$\nearrow$ &$9\pi^4\approx$876.6818\vspace{2mm}\\
$r$&---&1.702863&1.894823&1.946762&&\vspace{2mm}\\
 $\lambda_{2,h}$& 2720.0885  & 3226.6792   &3372.2667 &3428.9320 &$\nearrow$ & $36\pi^4\approx$3506.7273\vspace{2mm}\\
 $\lambda_{3,h}$& 2720.0885  & 3226.6792   &3372.2667 &3428.9320 &$\nearrow$ & $36\pi^4\approx$3506.7273\vspace{2mm}\\
 $\lambda_{4,h}$& 2720.0885  & 3226.6792   &3372.2667 &3428.9320 &$\nearrow$& $36\pi^4\approx$3506.7273\vspace{2mm} \\
$r$&---&1.490027&1.809503&1.902066&&\vspace{2mm}\\
 $\lambda_{5,h}$& 5246.9541  &6842.3245   &7369.5014 &7584.7868 &$\nearrow$& $81\pi^4\approx$7890.1364\vspace{2mm}\\
 $\lambda_{6,h}$& 5246.9541  &6842.3245   &7369.5014 &7584.7868 &$\nearrow$&$81\pi^4\approx$7890.1364\vspace{2mm}\\
 $r$&---&1.334896&1.724958&1.854797&&\\\hline
 \end{tabular}
\end{table}

\end{CJK*}
\end{document}